\documentclass[12pt,oneside,reqno]{amsart}

\usepackage[utf8]{inputenc}
\usepackage{amsmath}
\usepackage{amsthm, amsfonts,srcltx,amsopn}
 \usepackage{mathabx}
\usepackage{amssymb}
\usepackage[mathscr]{euscript}
\usepackage{enumerate}
\usepackage{graphicx}
\usepackage{fullpage}
\usepackage{tikz}
\usepackage{color}
\usepackage{multicol}
\usepackage{mathrsfs}
\usepackage{hyperref}
\usepackage{float}
\usepackage{subcaption}
\usepackage{calc}

\usetikzlibrary{arrows}
\usepackage{tikz-cd} 
\tikzcdset{row sep/normal=.1cm}
\tikzcdset{column sep/normal=.75cm}

\newtheorem{thm}{Theorem}[section]
\newtheorem{prop}[thm]{Proposition}
\newtheorem{lem}[thm]{Lemma}
\newtheorem{cor}[thm]{Corollary}
\newtheorem{conj}[thm]{Conjecture}

\theoremstyle{definition}
\newtheorem{defn}[thm]{Definition}
\theoremstyle{definition}
\newtheorem{ex}[thm]{Example}
\theoremstyle{definition}

\theoremstyle{definition}
\newtheorem{remark}[thm]{Remark}
\theoremstyle{definition}

\newcommand*\Z{\mathbb{Z}}

\newcommand*\Wit{\operatorname{Wit}}

\newcommand*\Sep{\operatorname{Sep}(S)}
\newcommand*\cusp{\operatorname{cusp}}

\newcommand*\diam{\operatorname{diam}}
\newcommand*\rank{\operatorname{rank}}
\newcommand*\nest{\sqsubseteq}
\newcommand*\propnest{\sqsubsetneq}

\newcommand*\mf[1]{\mathfrak{#1}}
\newcommand*\mc[1]{\mathcal{#1}}
\newcommand*\trans{\pitchfork}

\newcommand*\gate{\mathfrak{g}}

\newcommand{\cut}{\operatorname{Cut}}

\newcommand*\lk{\operatorname{lk}}
\newcommand*\st{\operatorname{st}}

\newcommand{\tsh}[1]{\left\{\kern-.7ex\left\{#1\right\}\kern-.7ex\right\}}
\newcommand{\Tsh}[2]{\tsh{#2}_{#1}}
\newcommand{\threshold}[2]{\Tsh{#2}{#1}}
\newcommand{\conespace}{\widehat{\mc{X}}}

\newcommand{\iso}{\mc{I}}

\title[From Hierarchical to Relative Hyperbolicity]{ 
From Hierarchical to Relative Hyperbolicity}
\author{Jacob Russell}
\date{\today}

\begin{document}	
\begin{abstract}
We provide a simple, combinatorial criteria for a hierarchically hyperbolic space to be relatively hyperbolic by proving a new formulation of relative hyperbolicity in terms of hierarchy structures. In the case of clean hierarchically hyperbolic groups, this criteria characterizes relative hyperbolicity. We  apply our criteria to graphs associated to surfaces and prove that the separating curve graph of a surface is relatively hyperbolic when the surface has zero or two punctures. We also recover a celebrated theorem of Brock and Masur on the relative hyperbolicity of the Weil--Petersson metric on Teichm\"uller space for surfaces with complexity three.
\end{abstract}

\maketitle

\section{Introduction}
Since Gromov's ground breaking treatise on hyperbolic groups \cite{gromov1}, considerable progress has been made towards generalizing Gromov's coarse negative curvature to accommodate spaces that have a mixture of negative and non-negative curvature. Gromov himself suggested the study of \emph{relatively hyperbolic groups} where the non-hyperbolic features are contained in a collection of isolated, peripheral subgroups. Following Gromov, Farb and Bowditch put forth a rich theory of relatively hyperbolic groups \cite{Farb_Rel_Hyp,Bowditch_Rel_Hyp}.  These ideas generalize profitably from group to general metric spaces, with a relatively hyperbolic space being hyperbolic outside of a collection of isolated, peripheral subspaces.  An expansive and fruitful study of the coarse geometry of relatively hyperbolic groups and spaces has cemented them as central examples in geometric group theory; see \cite{Osin_Rel_Hyp_Groups,Groves_Manning_Rel_Hyp_Dehn_Filling,Brock_Masur_WP_Low_complexity, Drutu_Sapir_Rel_hyp, SistoMetRel}.

Despite the success of relatively hyperbolic spaces, many of the groups and spaces that demonstrate a mixture of hyperbolic and non-hyperbolic geometry fail to be relatively hyperbolic.  Most famously, the mapping class group of a surface is not relatively hyperbolic \cite{behrstock_Drutu_Mosher_Thickness,AAS_Non_Rel_Hyp_Groups}. Despite this failure, a beautiful picture of the mixing of hyperbolic and non-hyperbolic geometry in the mapping class group has developed based on the subsurface projection machinery of Masur and Minsky; see \cite{MMI,MMII, behrstock_Drutu_Mosher_Thickness,Behrstock_Thesis,BKMM_Rigidity,BM_MCG_Rank}.

Recently, Behrstock, Hagen, and Sisto  axiomatized the subsurface projection machinery of Masur and Minsky by formulating the class of \emph{hierarchically hyperbolic spaces} (HHS) \cite{BHS_HHSI,BHS_HHSII}. In addition to the mapping class group, hierarchically hyperbolic spaces include virtually compact special groups \cite{BHS_HHSII}, the fundamental group of any $3$--manifold without Nil or Sol components \cite{BHS_HHSII}, graph products of hyperbolic groups \cite{BR_Combination}, and Teichm\"uller space with either the Teichm\"uller or Weil--Petersson metric \cite{BHS_HHSI,MMI,MMII, BKMM_Rigidity,Brock_WeilPetersson, Durham_Combinatorial_Teich,Rafi_Combinatorial_Teich, EMR_Teich_Rank}.  A hierarchically hyperbolic space is a pair $(\mc{X},\mf{S})$ where $\mc{X}$ is a quasi-geodesic metric space and $\mf{S}$ is set indexing a collection of uniformly hyperbolic spaces $\{CW: W \in \mf{S}\}$. For each $W \in \mf{S}$, there is a projection map $\mc{X} \rightarrow CW$. The HHS axioms characterize the image of $\mc{X}$  under these projection maps, and describe how to use the geometry of the spaces indexed by $\mf{S}$ to study the geometry of $\mc{X}$. 
 
 Conceptually, the spaces $\{CW:W \in \mf{S}\}$ should be viewed as the ``hyperbolic parts" of $\mc{X}$, and the axioms should be thought of as instructions for how these hyperbolic parts fit together inside of $\mc{X}$. In particular,  the axioms equip $\mf{S}$ with a relation called \emph{orthogonality} that encodes natural products in $\mc{X}$.  If $W,V \in \mf{S}$ are orthogonal, than the product map $\mc{X} \rightarrow CW \times CV$ is (coarsely) onto, and the HHS axioms allow the pulling back of the product structure of $CW \times CV$ to produce a product region in $\mc{X}$ itself.  The orthogonality relation therefore produces a notion of rank for a hierarchically hyperbolic space; the \emph{rank of an HHS} $(\mc{X},\mf{S})$ is the cardinality of the largest pairwise orthogonal subset of $\mf{S}$ such that the corresponding hyperbolic spaces all have infinite diameter.
 
 Behrstock, Hagen, and Sisto proved that rank 1 hierarchically hyperbolic spaces are hyperbolic (Theorem \ref{thm:hyperbolic_HHSs}). This shows that the product regions are similar to the peripheral subsets of relatively hyperbolic spaces in that they contain all of the non-hyperbolic behavior in an HHS.  However, unlike the case of a relatively hyperbolic space, the product regions in a hierarchically hyperbolic space do not need to be isolated and can have substantial interactions. The interactions between product regions are governed by two other relations on $\mf{S}$, a partial order called \emph{nesting} and a third relation called \emph{transversality}.
 
 The substance of this paper examines the relationship between hierarchically hyperbolic spaces and relatively hyperbolic spaces. In Theorem 9.1 of \cite{BHS_HHSII}, Behrstock, Hagen, and Sisto showed that if a space is hyperbolic relative to a collection of hierarchically hyperbolic  subsets, then the space is hierarchically hyperbolic. The main result of this paper is the following criteria for detecting when a hierarchically hyperbolic space is relatively hyperbolic.
 
\begin{thm}[Isolated orthogonality implies relative hyperbolicity]\label{thm:intro _isolated_orthogonality_implies_rel_hyp}
Let $(\mc{X},\mf{S})$ be a hierarchically hyperbolic space with the bounded domain dichotomy. If $\mf{S}$ has {isolated orthogonality}, then $\mc{X}$ is relatively hyperbolic.
\end{thm}
 
  The bounded domain dichotomy requires that the diameter of every element of $\{CW : W \in \mf{S}\}$ is either infinite or uniformly bounded. This is a mild regularity condition that is satisfied by every naturally occurring HHS. Roughly speaking, the set $\mf{S}$ has \emph{isolated orthogonality} if there exists a collection $\iso \subseteq \mf{S}$, so that whenever $W$ and $V$ are orthogonal, there exists a unique element of $\iso$ containing both $W$ and $V$ in the partial order on $\mf{S}$ (for this condition to be non-trivial, we require that $\iso$ does not contain the unique maximal element of the partial order on $\mf{S}$). 
  
 The strength of Theorem \ref{thm:intro _isolated_orthogonality_implies_rel_hyp}  lies in its ability to detect relative hyperbolicity solely from the combinatorial structure of the relations on the set $\mf{S}$ and without examining the geometry of $\mc{X}$ directly. Additionally, the proof of Theorem \ref{thm:intro _isolated_orthogonality_implies_rel_hyp} provides an explicit description of the peripheral subsets in terms of the HHS structure; see Theorem \ref{thm:isolated_orthogonality_implies_Relative_hyperbolicity}. These advantages are illustrated in the applications of Theorem \ref{thm:intro _isolated_orthogonality_implies_rel_hyp} in this paper as well as the recent use of Theorem \ref{thm:intro _isolated_orthogonality_implies_rel_hyp}  by Behrstock, Hagen, Martin, and Sisto to prove the relative hyperbolicity of specific quotients of the mapping class group of a genus two surface \cite{BHMS_HHS_combinatorial}.

  In the most prominent class of HHSs, clean hierarchically hyperbolic groups,  we  refine Theorem \ref{thm:intro _isolated_orthogonality_implies_rel_hyp} to characterize relative hyperbolicity.
 
 \begin{cor}[Characterization of relatively hyperbolic  clean HHGs]\label{cor:intro_relatively_hyperbolic_HHGs}
 A clean hierarchically hyperbolic group is relatively hyperbolic if and only if it admits a hierarchically hyperbolic group structure with isolated orthogonality.
 \end{cor}

 The essence of the proof of Theorem \ref{thm:intro _isolated_orthogonality_implies_rel_hyp} is that isolated orthogonality forces the various product regions of the HHS to be isolated. Since the product regions contain all of the non-hyperbolic behavior in the space, once the product regions are isolated, the space will be hyperbolic relative to these product regions.  In this sense, Theorem \ref{thm:intro _isolated_orthogonality_implies_rel_hyp} is an HHS version of the result of Hruska and Kleiner that a $\text{CAT}(0)$ space with isolated flats and a geometric group action is relatively hyperbolic \cite{Hruske_Kleiner_Isolated_Flats}. In fact,  the CAT$(0)$ result is an immediate corollary of Theorem \ref{thm:intro _isolated_orthogonality_implies_rel_hyp}  in the special case of compact special groups (or more generally the class of cubulated groups studied in \cite{Hagen_Susse_factor_system}).

 As a direct application of Theorem \ref{thm:intro _isolated_orthogonality_implies_rel_hyp}, we demonstrate the relative hyperbolicity of several graphs built from curves on surfaces.    
 Vokes proved a large class of graphs associated to surfaces are hierarchically hyperbolic spaces \cite{vokes}.  This collections contains a number of important graphs from the literature, including the pants graph used by Brock to study the coarse geometry of the Weil--Petersson metric on Teichm\"uller space \cite{Brock_WeilPetersson}; the separating curve graph used by Brendle and Margalit to study the Johnson kernel of the mapping class group \cite{Margalit_Brendle_Separating_Curve_Graph_Johnson_Kernel}; and the cut system graph use by Hatcher and Thurston to prove finite presentability of the mapping class group \cite{Hatcher_Thurston_complex}.  By applying Theorem \ref{thm:intro _isolated_orthogonality_implies_rel_hyp}, we obtain a sufficient condition for the class of graphs studied by Vokes to be relatively hyperbolic. This produces the following new result for the separating curve graph as well as new, short proofs of Theorem 1 of \cite{Brock_Masur_WP_Low_complexity} and Theorem 1.2 of \cite{Li_Ma_Hatcher-Thurston} in the cases of the pants graph and cut system graphs. 
 
\begin{thm}[Relative hyperbolicity of graphs of multicurves] \label{thm:intro_graphs_of_multicurves}
Let $S_{g,n}$ be a orientable surface of finite type with genus $g$ and $n$ punctures. The following graphs are relatively hyperbolic.
\begin{enumerate}[(i)]
    \item The separating curve graph of $S_{g,n}$ when $2g+n \geq 6$ and $n=0$ or $n=2$.
    \item The pants graph of $S_{g,n}$ when $3g-3+n = 3$ \cite[Theorem 1]{Brock_Masur_WP_Low_complexity}. 
    \item The cut system graph of $S_{2,0}$ \cite[Theorem 1.2]{Li_Ma_Hatcher-Thurston}.
\end{enumerate}
\noindent Further, in each of the above cases, the peripherals are quasi-isometric to the product of curve graphs of proper, connected subsurfaces of $S_{g,n}$.
\end{thm}

Unlike the previously known results for the pants graph and cut system graph, Theorem  \ref{thm:intro_graphs_of_multicurves} gives the first examples of an infinite family of graphs associated to surfaces that are relatively hyperbolic, but not hyperbolic.

The examples in Theorem \ref{thm:intro_graphs_of_multicurves} suggest a broader conjecture:  one of the graphs associated to surfaces studied by Vokes is relatively hyperbolic if and only if Vokes' hierarchy structure has isolated orthogonality.
This conjecture is a natural companion to Corollary 1.5 of \cite{vokes}, which characterizes hyperbolicity of certain graph associated to surfaces in terms of the hierarchy structure. In Section \ref{subsec:conjecture}, we give evidence from the literature in support of this conjecture  and verify it for the cut system graph.

We also apply Theorem \ref{thm:intro _isolated_orthogonality_implies_rel_hyp} to right-angled Coxeter groups, which were shown to be hierarchically hyperbolic groups by Behrstock, Hagen, and Sisto \cite{BHS_HHSI}. By verifying an HHG structure has isolated orthogonality, we recover a sufficient condition for a right-angled Coxeter group to be relatively hyperbolic that was first proved by Caprace.

\begin{thm}[{\cite[Theorem A$'$]{Caprace_Erratum_rel_hyp_coxeter}}]\label{thm:intro_rel_hyp_RACG}
	Let $W_\Gamma$ be a right-angled Coxeter group with defining graph $\Gamma$. Suppose there exists a collection $\mc{J}$ of proper, full, non-complete subgraphs of $\Gamma$ satisfying the following
	\begin{enumerate}[(i)]
		\item If two non-complete full subgraphs $\Lambda_1$ and $\Lambda_2$ form a join subgraph, then there exists $\Omega \in \mc{J}$ such that $\Lambda_1, \Lambda_2 \subseteq \Omega$. 
		\item If $\Omega_1, \Omega_2 \in \mc{J}$ and $\Omega_1 \neq \Omega_2$, then $\Omega_1 \cap \Omega_2$ is either empty or a complete graph.
		\item If $\Lambda$ is a full subgraph of $\Gamma$ that is not a complete graph and $\Lambda \subseteq \Omega$ for some $\Omega \in \mc{J}$, then $\lk(\Lambda) \subseteq \Omega$.
	\end{enumerate} 
	Then, $W_\Gamma$ is hyperbolic relative to the subgroups $\{W_\Omega : \Omega \in \mc{J}\}$.
\end{thm}

The main tool we develop to prove Theorem \ref{thm:intro _isolated_orthogonality_implies_rel_hyp} is the following relative version of the Behrstock, Hagen, and Sisto result that rank 1 HHSs are hyperbolic.
 
 \begin{thm}[Rank 1 relative HHSs are relatively hyperbolic]\label{thm:rank_1_relative_HHS_intro}
  Let $(\mc{X},\mf{S})$ be a  relative hierarchically hyperbolic space with the bounded domain dichotomy. If $(\mc{X},\mf{S})$ has rank 1, then $\mc{X}$ is relatively hyperbolic.
 \end{thm}
 
 A relative hierarchically hyperbolic space is a hierarchically hyperbolic space where spaces corresponding to the minimal elements of $\mf{S}$ are not required to be hyperbolic. Behrstock, Hagen, and Sisto showed that every relatively hyperbolic space is also a rank 1 relative hierarchically hyperbolic space with the bounded domain dichotomy; see Theorem \ref{thm:Relative_HHS_structure_on_rel_hyp}. Theorem \ref{thm:rank_1_relative_HHS_intro} characterizes relative hyperbolicity by establishing that these are the only such examples. In the context of groups, the hypothesis of the bounded domain dichotomy can be dropped.
 
 \begin{cor}[Hierarchy formulation of relatively hyperbolic groups.]\label{cor:intro_relative_hyperbolic_groups}
 A finitely generated group is relatively hyperbolic if and only if it admits a rank 1 relative hierarchically hyperbolic group structure.
 \end{cor}
 
  We begin by reviewing relatively hyperbolic and hierarchically hyperbolic spaces in Section \ref{sec:background}. In Section \ref{sec:rank_1_relative_HHSs}, we  prove our two results on relative HHSs, Theorem \ref{thm:rank_1_relative_HHS_intro} and Corollary \ref{cor:intro_relative_hyperbolic_groups}. In Section \ref{sec:relatively_hyperbolic_HHSs}, we use Theorem \ref{thm:rank_1_relative_HHS_intro} to prove our main result, Theorem \ref{thm:intro _isolated_orthogonality_implies_rel_hyp}.  The proof shows how to build a rank 1 relative HHS structure for every HHS with isolated orthogonality. Theorem \ref{thm:intro _isolated_orthogonality_implies_rel_hyp} then follows by applying Theorem \ref{thm:rank_1_relative_HHS_intro}.  Building on Theorem \ref{thm:intro _isolated_orthogonality_implies_rel_hyp}, we prove Corollary \ref{cor:intro_relatively_hyperbolic_HHGs} in Section \ref{sec:clean_hhgs}. The final two sections are devoted to applying Theorem \ref{thm:intro _isolated_orthogonality_implies_rel_hyp} to specific examples of hierarchically hyperbolic spaces. We recover Caprace's sufficient condition for a right-angled Coxeter group to be relatively hyperbolic in Section \ref{sec:RACG} and we prove the relative hyperbolicity of the separating curve graph, pants graph, and cut system graph in Section \ref{sec:complexes_of_curves}.\\

  \noindent \textbf{Acknowledgements:} The author would like to thank Kate Vokes for the many lively discussions that inspired much of the present work and for catching an error in an early draft of this paper. The author is also grateful to Alessandro Sisto for his insights into both hierarchically hyperbolic and relatively hyperbolic spaces as well as his assistance in acquiring references for this paper. The author is very grateful for the detailed comments of an anonymous referee that greatly improved the exposition of this paper.  Finally, the author gives thanks to his PhD advisor, Jason Behrstock, for his  mentorship during this project and for providing helpful feedback on several drafts of this paper.

\section{Background}\label{sec:background}
\subsection{Relative Hyperbolicity}
We begin by reviewing the definition of a relatively hyperbolic space that we shall utilize.  We present the definition in terms of combinatorial horoballs inspired by \cite{Bowditch_Rel_Hyp} and \cite{Groves_Manning_Rel_Hyp_Dehn_Filling}.  The equivalence between this definition and several others is proven in \cite{SistoMetRel}. We start with the underlining objects, {quasi-geodesic spaces}.

\begin{defn}
A metric space $X$ is a \emph{$(K,C)$--quasi-geodesic space} if for all $x,y \in X$ there exists a $(K,C)$--quasi-geodesic $\gamma \colon [a,b] \rightarrow X$ with $\gamma(a) = x$ and $\gamma(b) = y$.
\end{defn}

Given a $(K,C)$--quasi-geodesic  space $X$, we can construct a geodesic space quasi-isometric to $X$ by fixing an  $\epsilon$--separated net  $\Gamma \subseteq X$ and connecting pairs of points $x,y \in \Gamma$ by an edge of length $d(x,y)$ whenever $d(x,y) < 2\epsilon$. The resulting metric graph is quasi-isometric to $X$. We shall call this metric graph the \emph{$\epsilon$--approximation graph of $X$} and denote it by $\Gamma(X)$. Since $\epsilon$ can be chosen to depend only on $K$ and $C$, $\Gamma(X)$ can be constructed such that the quasi-isometry constants also depend only on $K$ and $C$.  

When studying hierarchically hyperbolic spaces, the necessity of working with quasi-geodesic instead of geodesic spaces arises from naturally occurring subsets that are quasi-geodesic, but not geodesic when equipped with the induced metric; see Section \ref{subsec:product_regions}.

A particularly important class of quasi-geodesic spaces are the \emph{hyperbolic spaces} introduced by Gromov \cite{gromov1}.

\begin{defn}\label{def:hyperbolic}
A $(K,C)$--quasi-geodesic metric space is \emph{$\delta$--hyperbolic} if for every $(K,C)$--quasi-geodesic triangle the  $\delta$--neighborhood of the union of any two of the sides contains the third. 
\end{defn}

To define a relatively hyperbolic space, we first need to define the combinatorial horoball over a metric space.\footnote{The combinatorial horoballs presented here are quasi-isometric to the horoballs defined in \cite{Groves_Manning_Rel_Hyp_Dehn_Filling}.}

\begin{defn}[Combinatorial horoball]
If $X$ is a metric space and $\Gamma$ is an $\epsilon$--separated net for $X$, then the \emph{combinatorial horoball based on $\Gamma$} is the metric graph with vertices $\Gamma \times \mathbb{N}$ and edges of the form:
\begin{itemize}
    \item For all $n\in\mathbb{N}$ and $x \in \Gamma$, $(x,n)$ and $(x,n+1)$ are connected by an edge of length $1$.
    \item For all $x,y \in \Gamma$, $(x,n)$ and $(y,n)$ are connected by an edge of length $e^{-n}d_X(x,y)$.
\end{itemize}

The \emph{combinatorial horoball over $X$ based on $\Gamma$} is the union of $X$ with the combinatorial horoball based on $\Gamma$  and given the induced metric.
\end{defn}

The following is the main lemma we shall need about the geometry of horoballs.  Here, and in the sequel, the notation $A \stackrel{K,C}{\asymp} B$ denotes $(B-C)/K \leq A \leq KB+C$.

\begin{lem}[{\cite{MacKaySisto,Groves_Manning_Rel_Hyp_Dehn_Filling}}]\label{lem:horoball_distance}
Let $X$ be a $(K,C)$--quasi-geodesic space and $\mc{H}(X)$ be the horoball over $X$ based on an $\epsilon$--separated net $\Gamma$.

\begin{enumerate}[(i)]
    \item There exists $\delta$ depending on $\epsilon$, $K$, and $C$ such that $\mc{H}(X)$ is $\delta$--hyperbolic.
    \item There exists $L\geq 1$ depending only on $\epsilon$, $K$, and $C$ such that
\[\log(d_X(x,y)) \stackrel{L,L}{\asymp} d_{\mc{H}(X)}(x,y)\]
for all $x,y \in X$. 
\end{enumerate}
\end{lem}

If $M$ is a cusped hyperbolic $3$--manifold, then the space obtained by attaching a combinatorial horoball to each of the $\Z^2$ subgroups of $\pi_1(M)$ is quasi-isometric to $\mathbb{H}^3$.  This construction motivates the following definition of a relatively hyperbolic space.

\begin{defn}[Relatively hyperbolic space]\label{defn:relatively_hyperbolic_space}
Let $X$ be a quasi-geodesic space and $\mc{P}$ be a collection of uniformly coarsely connected\footnote{A subset $Y$ of a  metric space $X$ is \emph{coarsely connected} if there exists $C>0$ such that for all $x,y \in Y$, there exists a sequence of points $x=x_0,x_1,\dots,x_n = y$ in $Y$ with $d(x_{i-1},x_i) \leq C$ for all $1\leq i \leq n$.} subsets of $X$ with $d_{Haus}(X,P)=\infty$ for all $P \in \mc{P}$.   For each $P\in\mc{P}$, fix an $\epsilon$--separated net $\Gamma_P$. The \emph{cusped space}, $\cusp(X, \mc{P})$, is the metric  space obtained from attaching the combinatorial horoball on $\Gamma_P$ to $X$ for every $P \in \mc{P}$. The space $X$ is \emph{relatively hyperbolic with respect to $\mc{P}$} if $\cusp(X,\mc{P})$ is hyperbolic. In this case, we call the subsets in $\mc{P}$ the \emph{peripheral subsets} of $X$.
\end{defn}

Theorem 1.1 of \cite{SistoMetRel} established that Definition \ref{defn:relatively_hyperbolic_space} is equivalent to several other formulation of relative hyperbolicity, including those in terms of asymptotically tree graded spaces and the bounded subset penetration property. Sisto also shows that the hyperbolicity of the cusped space is independent of the choice of nets for each the peripheral subsets.

An important class of relatively hyperbolic spaces are relatively hyperbolic groups.

\begin{defn}[Relatively hyperbolic group]
Let $G$ be a finitely generated group and $H_1,\dots, H_n$ be a finite collection of finitely generated subgroups. We say $G$ is \emph{hyperbolic relative to} $H_1,\dots,H_n$ if the Cayley graph of $G$, with respect to a finite generating set, is hyperbolic relative the collection of left cosets of $H_1,\ldots, H_n$.
\end{defn}
 
A priori, the above definition of a relatively hyperbolic group is stronger than simply requiring the Cayley graph of the group to be  hyperbolic relative to some collection of peripheral subsets. However, the following theorem of Drutu established that the relative hyperbolicity of a group is equivalent to the metric relative hyperbolicity of the Cayley graph.

\begin{thm}[{\cite[Theorem 1.5]{Drutu_Rel_Hyp_is_Geometric}}]\label{thm:rel_hyp_is_geometric}
Let $G$ be a finitely generated group and $X$ be the Cayley graph of $G$ with respect to some finite generating set.  If $X$ is hyperbolic relative to a collection $\mc{P}$ of coarsely connected subsets, then $G$ is hyperbolic relative to some subgroups $H_1,\dots, H_n$ where each $H_i$ is contained in a regular neighborhood of an element of $\mc{P}$.
\end{thm}

\subsection{Hierarchically Hyperbolic Spaces}
We now recall the definition of a (relative) hierarchically hyperbolic space as well as some of the basic terminology and tools for working with HHSs. The definition presented here is the variant of the HHS axioms discussed in Section 1.3 of \cite{BHS_HHSII}. They are equivalent to the original axioms by Proposition 1.11 and Remark 1.3 of \cite{BHS_HHSII}.

\begin{defn}[Relative hierarchically hyperbolic space]\label{defn:HHS}
Let $\mc{X}$ be a quasi-geodesic space. A \emph{relative hierarchically hyperbolic structure} (relative HHS structure) on $\mc{X}$ consists of a constant $E>0$, an index set $\mathfrak S$, and a set $\{ C W : W\in\mathfrak S\}$ of geodesic spaces $( C W,d_W)$  such that the following axioms are satisfied. 
\begin{enumerate}[(i)]
\item\textbf{(Projections.)}\label{axiom:projections} For each $W \in \mf{S}$, there exists a \emph{projection} $\pi_W \colon \mc{X} \rightarrow 2^{CW}$ such that for all $x \in \mc{X}$, $\pi_W(x) \neq \emptyset$ and $\diam(\pi_W(x))<E$. Moreover, each $\pi_W$ is $(E, E)$--coarsely
Lipschitz and $CW \subseteq \mc{N}_E(\pi_{W}(\mc{X}))$.

\item\textbf{(Uniqueness.)} For each $\kappa\geq 0$, there exists
$\theta=\theta(\kappa)$ so that if $x,y\in\mc X$ and
$d_\mc{X}(x,y)\geq\theta$, then there exists $W\in\mathfrak S$ with $d_W(x,y)\geq \kappa$.\label{axiom:uniqueness}

 \item \textbf{(Nesting.)} \label{axiom:nesting} If $\mathfrak S \neq \emptyset$, then $\mf{S}$ is equipped with a  partial order $\nest$ and contains a unique $\nest$--maximal element. When $V\nest W$, we say $V$ is \emph{nested} in $W$.  For each
 $W\in\mathfrak S$, we denote by $\mathfrak S_W$ the set of all $V\in\mathfrak S$ with $V\nest W$.  Moreover, for all $V,W\in\mathfrak S$ with $V \propnest W$ there is a specified non-empty subset $\rho^V_W\subseteq C W$ with $\diam(\rho^V_W)\leq E$.

 \item \textbf{(Hyperbolicity.)} \label{axiom:hyperbolicity} For each $W \in \mf{S}$, either $W$ is $\nest$--minimal or $CW$ is $E$--hyperbolic.

 \item \textbf{(Finite complexity.)} \label{axiom:finite_complexity} There exists $n\geq0$ so that any set of pairwise $\nest$--comparable elements has cardinality at most $n$.

 \item \textbf{(Large links.)} \label{axiom:large_link_lemma} 
For all $W\in\mathfrak S$ and  $x,y\in\mc X$, there exists $\mf{L} = \{V_1,\dots,V_m\}\subseteq\mathfrak S_W-\{W\}$ such that $m$ is at most $E d_{W}(\pi_W(x),\pi_W(y))+E$, and for all $U\in\mathfrak
S_W-\{W\}$, either $U\in\mathfrak S_{V_i}$ for some $i$, or $d_{U}(\pi_V(x),\pi_V(y))<E$.

 \item \textbf{(Bounded geodesic image.)} \label{axiom:bounded_geodesic_image} For all $x,y \in \mc{X}$ and $V,W\in\mathfrak S$ with $V \propnest W$ if $d_V\bigl(\pi_V(x),\pi_V(y)\bigr)\geq E$, then every $CW$ geodesics from $\pi_W(x)$ to $\pi_W(y)$ must intersect the $E$--neighborhood of $\rho_W^V$.

 \item \textbf{(Orthogonality.)} 
 \label{axiom:orthogonal} $\mathfrak S$ has a symmetric relation called \emph{orthogonality}. If $V$ and $W$ are orthogonal, we write $V\perp
 W$ and require that $V$ and $W$ are not $\nest$--comparable. Further, whenever $V\nest W$ and $W\perp
 U$, we require that $V\perp U$. We denote by $\mf{S}_W^\perp$ the set of all $V\in \mf{S}$ with $V\perp W$.
 
 \item \textbf{(Containers.)} \label{axiom:containers}  For each $W \in \mf{S}$ and $U \in \mf{S}_W$ with $ \mf{S}_W\cap \mf{S}_U^\perp \neq \emptyset$, there exists $Q \in\mf{S}_W$ such that $V \nest Q$ whenever $V \in\mf{S}_W \cap \mf{S}_U^\perp$.  We call $Q$ the \emph{container of $U$ in $W$}.
 
 \item \textbf{(Transversality and consistency.)}
 \label{axiom:consistency} If $V,W\in\mathfrak S$ are not
 orthogonal and neither is nested in the other, then we say $V,W$ are
 \emph{transverse}, denoted $V\trans W$.  If $V\trans W$, then there are non-empty
  sets $\rho^V_W\subseteq C W$ and
 $\rho^W_V\subseteq C V$ each of diameter at most $E$ and 
 satisfying $\min\left\{d_{
 W}(\pi_W(x),\rho^V_W),d_{
 V}(\pi_V(x),\rho^W_V)\right\}\leq E$ for all $x\in\mc X$.

 If $U\nest V$ and either $V\propnest W$ or $V\trans W$ and $W\not\perp U$, then $d_W(\rho^U_W,\rho^V_W)\leq E$.

 \item \textbf{(Partial realization.)} \label{axiom:partial_realization}  If $\{V_i\}$ is a finite collection of pairwise orthogonal elements of $\mathfrak S$ and $p_i\in  C V_i$, then there exists $x\in \mc X$ so that:
 \begin{itemize}
 \item $d_{V_i}(x,p_i)\leq E$ for all $i$,
 \item for each $i$ and 
 each $W\in\mathfrak S$, if $V_i\propnest W$ or $W\trans V_i$, we have 
 $d_{W}(x,\rho^{V_i}_W)\leq E$.

 \end{itemize}

\end{enumerate}

We use $\mf{S}$ to denote the entire relative HHS structure, including all the spaces,  projections, and relations defined above. If for every $W \in \mf{S}$, $CW$ is $E$--hyperbolic, then $\mf{S}$ is a \emph{hierarchically hyperbolic structure} on $\mc{X}$. We call a quasi-geodesic space $\mc{X}$ a  (relative) hierarchically hyperbolic space if there exists a (relative) hierarchically hyperbolic structure on $\mc{X}$. We use the pair $(\mc{X},\mf{S})$ to denote a (relative) hierarchically hyperbolic space equipped with the specific (relative) HHS structure $\mf{S}$.

\end{defn}

\begin{remark}[Normalized hierarchy structures]
 The requirement in Axiom \ref{axiom:projections} that the maps $\pi_W$ are coarsely surjective means that the above definition is technically that of a \emph{normalized relative hierarchically hyperbolic space}. The definition of a non-normalized relative hierarchically hyperbolic space is identical except we only require $\pi_W(\mc{X})$ to  be uniformly quasiconvex in $CW$ instead of coarsely covering for all $W \in \mf{S}$ with $CW$ hyperbolic.  Given a non-normalized relative hierarchically hyperbolic space $(\mc{X},\mf{S})$, we can replace each $CW$ with the geodesic thickening of $\pi_W(\mc{X})$ to produced a normalized HHS structure on $\mc{X}$ with index set $\mf{S}$. See  Remark 1.3 of \cite{BHS_HHSII}  or Proposition 1.16 of \cite{HHS_boundary} for details.  While we will operate under the standing assumption that the all of our relative hierarchically hyperbolic spaces are normalized, our results will continue to hold in the non-normalized setting. In this case, the first step in each proof will be to normalize the structure as described above.
\end{remark}

\begin{remark}[Large links simplification]
 In \cite{BHS_HHSII}, the large links axiom (Axiom \ref{axiom:large_link_lemma}) additionally required that $d_{W}(\pi_W(x),\rho^{V_i}_W)\leq E d_{W}(\pi_W(x),\pi_W(y))+E$ for each $i$.  We have omitted this statement from the definition here as it can be derived, after possibly increasing $E$, from the bounded geodesic image axiom (Axiom \ref{axiom:bounded_geodesic_image}) and the last line of the consistency axiom (Axiom \ref{axiom:consistency}). This simplification does not require that the hierarchy structure be normalized.
\end{remark}

Given a relative hierarchically hyperbolic space $(\mc{X},\mf{S})$, we call the elements of $\mf{S}$ \emph{domains} and the associated spaces, $CW$, \emph{shadow spaces}. We use $\mf{S}^{rel}$ to denote the collection of domains whose shadow spaces are not hyperbolic. We say a domain $W\in\mf{S}$ is \emph{infinite} if $\diam(CW) = \infty$. If $V\trans W$ or $V \propnest W$, then the subset $\rho_W^V$ is called the \emph{relative projection from $V$ to $W$} and should be viewed as a coarsely constant map $\rho_V^W \colon CW \rightarrow CV$. We call the constant $E$ the \emph{hierarchy constant for $(\mc{X},\mf{S})$}.

\begin{remark}[Moving an HHS structure over a quasi-isometry]\label{rem:HHS_and_QI}
 If $(\mc{X},\mf{S})$ is a relative hierarchically hyperbolic space and $f \colon \mc{Y} \to \mc{X}$ is a quasi-isometry, then $\mc{Y}$ admits a relative HHS structure with same index set, shadow spaces, relations and relative projections as $\mf{S}$ where the projection maps are given by $\pi_W \circ f$ for each $W \in \mf{S}$. Thus, when proving $\mf{S}$ is a relative HHS structure on $\mc{X}$, we can assume $\mc{X}$ is a metric graph instead of a quasi-geodesic space by using an approximation graph for $\mc{X}$.
\end{remark}

When writing distances in the shadow spaces, we often simplify the notation by suppressing the projection map $\pi_W$. That is, given $x,y\in\mc X$ and $p\in C W$  we write $d_W(x,y)$ for $d_W(\pi_W(x),\pi_W(y))$ and $d_W(x,p)$ for $d_W(\pi_W(x),p)$. When we measure distance between a pair of sets  we are taking the minimum distance between the two sets.

While the definition of a relative HHS requires the shadow spaces to be geodesic spaces, it is a straight froward exercise to check that it is sufficient to only require that the shadow spaces are all uniformly quasi-geodesic spaces.  This is a useful simplification when trying to define new HHS structures utilizing subsets of an existing hierarchically hyperbolic space.

\begin{lem}[Quasi-geodesic shadow spaces]\label{lem:quasi-geodesic_shadow_spaces}
Suppose a quasi-geodesic space $\mc{X}$ and an index set $\mf{S}$ satisfy the condition of a (relative) hierarchically hyperbolic space structure in Definition \ref{defn:HHS} with the following changes.
\begin{itemize}
    \item There exist $K\geq 1$ and $C\geq 0$, such that for all $W \in\mf{S}$, the shadow space $CW$ is a $(K,C)$--quasi-geodesic space instead of a geodesic space.
    \item Replace Axiom \ref{axiom:bounded_geodesic_image} with: For all $x,y\in\mc{X}$ and $W,V \in \mf{S}$ with $V \propnest W$ if $d_V(x,y)>E$, then every $(K,C)$--quasi-geodesic in $CW$  from $\pi_W(x)$ to $\pi_W(y)$  must intersect the $E$--neighborhood of $\rho_W^V$.
\end{itemize}
Then $(\mc{X},\mf{S})$ is a (relative) hierarchically hyperbolic space where the shadow space for each $W\in\mf{S}$ is $\Gamma({CW})$, the approximation graph of $CW$; the projection maps are $f_W \circ \pi_W$ where $f_W$ is the uniform quasi-isometry $CW \rightarrow \Gamma(CW)$; and the relative projections are $f_V(\rho_W^V)$ whenever $V\trans W$ or $V\propnest W$.
\end{lem}

\begin{proof}
	Since each $CW$ is a $(k,c)$--quasi-geodesic space, each $CW$ is uniformly quasi-isomteric to its approximation graph $\Gamma(CW)$. Thus, all of the HHS axioms, except the bounded geodesic image axiom, hold with a uniform increase in the hierarchy constant after replacing $CW$ with $\Gamma(CW)$. The bounded geodesic image axiom follows from the second bullet plus that fact that all quasi-geodesics in a hyperbolic space are uniformly close to the geodesic between their endpoints; see \cite[Theorem III.H.1.7]{BH}.
\end{proof}

The central premise of the study of relative hierarchically hyperbolic spaces is that the geometry of $\mc{X}$ can be recovered from the geometry of the shadow spaces. One way of doing so is through hierarchy paths, quasi-geodesic in $\mc{X}$ that project to uniform quality quasi-geodesics in each of the shadow spaces.

\begin{defn}[Hierarchy path]\label{defn:hierarchy_path}
A $(\lambda,\lambda)$--quasi-geodesic $\gamma$ in a relative hierarchically hyperbolic space $(\mc{X},\mf{S})$ is a \emph{$\lambda$--hierarchy path} if for each $W \in \mf{S}$,  the map $\pi_W \circ \gamma$ is an unparameterized $(\lambda,\lambda)$--quasi-geodesic.\footnote{ A map \(f \colon [a,b] \to X\) is an unparameterized \((\lambda, \lambda)\)--quasi-geodesic if there exists  an increasing function \(g \colon [0, \ell] \to [a, b]\) such that $g(0)=a$, $g(\ell)=b$, and \(f \circ g \) is a \((\lambda, \lambda)\)--quasi-geodesic of \(X\).}
\end{defn}

\begin{thm}[The existence of hierarchy paths. {\cite[Theorem 6.11]{BHS_HHSII}}] \label{thm:hierarchy_paths_exist}
If $(\mc{X},\mf{S})$ is a relative hierarchically hyperbolic space, then there exists $\lambda_0 \geq 1$ such that every pair of points in $\mc{X}$ can be joined by a $\lambda_0$--hierarchy path.
\end{thm}

For the results in this paper, we are going to restrict our attention to relative hierarchically hyperbolic spaces with the following regularity condition.

\begin{defn}[Bounded domain dichotomy]
A relative hierarchically hyperbolic space $(\mc{X},\mf{S})$ has the \emph{bounded domain dichotomy} if there exists $B>0$ such that for all $U \in \mf{S}$, if $\diam(CU) > B$, then $\diam(CU) = \infty$.
\end{defn}

Every naturally occurring example of a relative HHS satisfies the bounded domain dichotomy. In particular, it is the consequence of the following definition of a \emph{relative hierarchically hyperbolic group} that every such group has the bounded domain dichotomy.

\begin{defn}[Hierarchically hyperbolic group]\label{defn:hierarchically hyperbolic groups}
  
     Let $G$ be a finitely generated group and $X$ be the Cayley graph of $G$ with respect to some finite generating set.  We say $G$ is a (relative) \emph{hierarchically hyperbolic group} (HHG) if:
    
   \begin{enumerate}[(i)]
       \item The space $X$ admits an (relative) HHS structure $\mf{S}$ with hierarchy constant $E$.\label{item:HHG_cayley_graph_is_HHS}
        \item There is a $\nest$, $\perp$, and $\trans$ preserving action of $G$ on $\mf{S}$ by bijections such that $\mf{S}$ contains finitely many $G$--orbits.\label{item:HHG_action_on_S}
        \item \label{item:HHG_equivariance} For each $W \in \mf{S}$ and $g\in G$, there exists an isometry $g_W \colon CW \rightarrow C(gW)$ satisfying the following for all $V,W \in \mf{S}$ and $g,h \in G$.
      \begin{itemize}
        \item The map $(gh)_W \colon CW \to C(ghW)$ is equal to the map $g_{hW} \circ h_W \colon CW \to C(hW)$.
        \item For each $x \in X$, $g_W(\pi_W(x))$ and $\pi_{gW}(g \cdot x)$ are at most $E$ far apart in $C(gW)$.
        \item If $V \trans W$ or $V \propnest W$, then $g_W(\rho_W^V)$  and $\rho_{gW}^{gV}$ are at most $E$ far apart in $C(gW)$.
        \end{itemize}
        \end{enumerate}

The structure $\mf{S}$ satisfying (\ref{item:HHG_cayley_graph_is_HHS})-(\ref{item:HHG_equivariance})  is called a (relative) hierarchically hyperbolic group (HHG) structure on $G$. We use $(G,\mf{S})$ to denote a group $G$ equipped with a specific (relative) HHG structure $\mf{S}$. 
\end{defn}

One of the benefits of the bounded domain dichotomy is that it allows  the strongly quasiconvex subsets of a hierarchically hyperbolic space to be understood in terms of their projections to the shadow spaces. Given a function $Q \colon [1,\infty) \to [0,\infty) \to \mathbb{R}$, we say a subset $Y$ of a quasi-geodesic space $X$ is \emph{$Q$--strongly quasiconvex} if for each $K\geq 1$, $C \geq 0$, there exists $Q(K,C) \geq 0$ such that every $(K,C)$--quasi-geodesic with endpoints on $Y$ is contained in the $Q(K,C)$--neighborhood of $Y$. The following characterizes strongly quasiconvex subsets of an HHS utilizing the hierarchy structure.

\begin{thm}[Quasiconvexity in HHSs; {\cite[Theorem 6.3]{RST_Quasiconvexity}}]\label{thm:quasiconvexity}
Let $(\mc{X},\mf{S})$ be an HHS with the bounded domain dichotomy. A subset $\mc{Y}$ of $\mc{X}$ is $Q$--strongly quasiconvex if and only if the following conditions are satisfied:
\begin{itemize}
    \item ({Hierarchical quasiconvexity}) There exists $R\colon [1,\infty) \rightarrow [0,\infty)$ such that every $\lambda$--hierarchy path with endpoints on $\mc{Y}$ is contained in the $R(\lambda)$--neighborhood of $\mc{Y}$.
    \item ({Orthogonal projection dichotomy}) There exists $B>0$ such that for all $U,V \in\mf{S}$ with $U \perp V$, if $\diam(\pi_U(\mc{Y})) >B$, then $CV \subseteq \mc{N}_B(\pi_V(\mc{Y}))$.
\end{itemize}
Further, the function $Q$ and the pair $(R,B)$ each determine the other.
\end{thm} 

\subsection{Hierarchy Structures on Relatively Hyperbolic Spaces}

We now recall hierarchy structures for relatively hyperbolic spaces describe by Behrstock, Hagen, and Sisto in Section 9 of \cite{BHS_HHSII}. These structures motivate the arguments in Sections \ref{sec:relatively_hyperbolic_HHSs} and \ref{sec:clean_hhgs}. 

\begin{thm}[Relatively hyperbolic spaces are relative HHSs. {\cite[Theorem 9.3]{BHS_HHSII}}] \label{thm:Relative_HHS_structure_on_rel_hyp}
If the quasi-geodesic space $X$ is hyperbolic relative to a collection of peripheral subsets $\mc{P}$, then $X$ admits a relative HHS structure $\mf{S}$ as follows.
\begin{itemize}
    \item The index set is $\mf{S} = \mc{P} \cup \{R\}$.
    \item The shadow space for $R$ is the space obtained from $X$  by attaching an edge of length 1 between every pair of points in $P$ for each $P \in \mc{P}$. The projection map $\pi_R$ is the inclusion map.
    \item The shadow space for $P \in \mc{P}$ is the subset $P$ and the projection map $\pi_P$ is the coarse closest point projection onto $P$ in $X$.
    \item $R$ is the $\nest$--maximal element of $\mf{S}$ and all other elements are transverse. For $P,Q \in \mc{P}$, the relative projection $\rho_P^Q$ is $\pi_P(Q)$ and the relative projection $\rho_R^P$ is the  now diameter 1 subset $P$ in $X$.
\end{itemize}
Further, if $G$ is a group that is  hyperbolic relative to finite collection of subgroups, then the above is a relative HHG structure on $G$.
\end{thm}

\begin{thm}[Hyperbolic relative to HHSs. {\cite[Theorem 9.3]{BHS_HHSII}}] \label{thm:hyperbolic_relative_HHS}
Let $X$ be hyperbolic relative to a collection of peripheral subsets $\mc{P}$. If for each $P \in \mc{P}$, $P$ admits an HHS structure $\mf{S}_P$ with hierarchy constant $E$, then $X$ admits an HHS structure $\mf{S}$ as follows.
\begin{itemize}
    \item The index set is $\mf{S} = \{R\} \cup \{\mf{S}_P\}_{P \in \mc{P}}$.
    \item The shadow space for $R$ is the space obtained from $X$  by attaching an edge of length 1 between every pair of points in $P$ for each $P \in \mc{P}$. The projection map $\pi_R$ is the inclusion map.
    \item For each $P \in \mc{P}$ and $U \in \mf{S}_P$, the shadow space for $U$ is the shadow space of $U$ in $(P,\mf{S}_P)$. The projection map $X \rightarrow CU$ is the composition of the closest point projection of $X$ onto $P$ with the projection map from $P$ onto $CU$ in $(P,\mf{S}_P)$.
    \item $R$ is the $\nest$--maximal element of $\mf{S}$. For each $P\in \mc{P}$ and $U,V \in \mf{S}_P$, $U$ and $V$ maintain the same relation and relative projections as in $\mf{S}_P$. For each $U \in \mf{S}_P$, the relative projection $\rho_R^U$ is the now diameter 1 subset $P$ in $X$. For $P,Q \in \mc{P}$,  $U \in \mf{S}_P$, and $ V \in \mf{S}_Q$, $U \trans V$ and the relative projection $\rho_V^U$ is the image of $P$ under the projection of $X \rightarrow CV$.
\end{itemize}

\noindent Further, if $G$ is a group that is  hyperbolic relative to a finite collection of subgroups $\{H_i\}$ and each $H_i$ is a hierarchically hyperbolic group, then the above is an HHG structure on $G$.
\end{thm}

\subsection{Standard Product Regions}\label{subsec:product_regions}

We now describe rigorously the product regions in  hierarchically hyperbolic spaces discussed in the introduction.

\begin{defn}[Product region]\label{defn:product_region}
Let $(\mc{X},\mf{S})$ be a  relative HHS and $U \in \mf{S}$.  The \emph{standard product region for $U$} is the set
\[ P_U = \left\{ x \in \mc{X} : d_V(x,\rho_V^U) \leq E \text{ for all } V \in \mf{S} \text{ with } U\trans V \text{ or } U \propnest V\right\}\]
Note, if $S$ is the $\nest$--maximal domain of $\mf{S}$, then $P_S = \mc{X}$.
\end{defn}

Each product regions  inherits a relative hierarchically hyperbolic structure from the ambient space as described in the next proposition.  Note, if the domain $V$ is neither nested into or orthogonal to the domain $U$, then $\pi_V(P_U)$ is bounded diameter as it is contained in the $E$--neighborhood of the relative projection $\rho_V^U$.

\begin{prop}[{Hierarchy structure on product regions; \cite[Proposition 5.6]{BHS_HHSII}}]\label{prop:product_regions_are_HHS}
Let $(\mc{X},\mf{S})$ be a  relative HHS and $U \in \mf{S}$.  Recall $\mathfrak S_U=\{V\in\mathfrak S : V\nest U\}$ and  $\mathfrak S_U^\perp=\{V\in\mathfrak S : V\perp U\}$. The set $P_U$, endowed with the subspace metric, is a $(K,C)$--quasi-geodesic space with $K$ and $C$ depending only on $(\mc{X},\mf{S})$. Further, $P_U$ inherits a  relative HHS structure from $(\mc{X},\mf{S})$ as follows.
\begin{itemize}
    \item The index set is $\mf{S}$. The relations between domains for $(P_U,\mf{S})$ are the same as they are for $(\mc{X},\mf{S})$.
    \item For domains in $\mf{S}_U \cup \mf{S}_U^\perp$, the shadow spaces, projection maps, and relative projections in $(P_U,\mf{S})$ are the same as they are in $(\mc{X},\mf{S})$.
    
    \item If $V \not\in \mf{S}_U \cup \mf{S}_U^{\perp}$ and $\pi_V$ is the projection for $V$ in $(\mc{X},\mf{S})$, then the shadow space for $V$ in $(P_U,\mf{S})$ will be the uniformly bounded diameter set $\pi_V(P_U)$ instead of the entire shadow space of $V$ in $(\mc{X},\mf{S})$. The projection map $P_U \to \pi_V(P_U)$ will be $\pi_V \vert_{P_{U}}$.  If $W \in \mf{S}$ with  $W\trans V$ or $W \propnest V$, then the relative projection for $W$ to $V$ in $(P_U,\mf{S})$ is all of $\pi_V(P_U)$. 
\end{itemize}

\end{prop}

Every relative hierarchically hyperbolic space satisfies a ``Masur--Minsky style" distance formula in which distances in the space are approximated by distance in the shadow spaces.  As each of the product regions inherits the relative HHS structure of the ambient space, we can formulate the distance formula so that the constants can be chosen uniformly for $\mc{X}$ and each of its product regions. When working with the distance formula, we adopt the notation $\threshold{N}{\sigma} = N$ if $N \geq \sigma$ and $\threshold{N}{\sigma} =0$ if $N<\sigma$.

\begin{thm}[The distance formula. {\cite[Theorem 6.10]{BHS_HHSII}}] \label{thm:distance_formula}
Let $(\mc X, \mathfrak S)$ be a   relative hierarchically hyperbolic space.  There exists $\sigma_0$ such that for all $\sigma \geq \sigma_0$, there exist $K\geq 1$, $C\geq 0$ so
that for any $U \in\mf{S}$,
$$d_{P_U}(x,y)\stackrel{{K,C}}{\asymp}\sum_{W\in\mathfrak
S_U \cup \mf{S}_U^\perp}\threshold{d_W(x,y)}{\sigma}$$ 
for any $x,y \in P_U$.
\end{thm}
 
If $S$ is the $\nest$--maximal element of $\mf{S}$, then $\mf{S}_S = \mf{S}$, $\mf{S}_S^\perp = \emptyset$, and $P_S =\mc{X}$, and we have the usual formulation of the distance formula in $(\mc{X},\mf{S})$. The number $\sigma$ determining the cut-off for $\threshold{d_W(x,y)}{\sigma}$ in the distance formula is called the \emph{threshold for the distance formula}.

While Definition \ref{defn:product_region} is concrete and succinct, it does little to illustrate the product structure underling the standard product regions. In fact, $P_U$ naturally decomposes into the product of two of relative hierarchically hyperbolic spaces as follows. Fix $p_0 \in P_U$ and define  two subsets of $P_U$: \begin{align*}
F_U &= \{x \in P_U : d_V(x,p_0)\leq E \text{ for all } V \in \mf{S}_U^\perp\} \\
E_U &= \{x \in P_U : d_V(x,p_0)\leq E \text{ for all } V \in \mf{S}_U\}. 
\end{align*} 
Proposition 5.11 of \cite{BHS_HHSII} shows that $F_U$ and $E_U$, equipped with the subspace metric, each inherit a relative HHS structure from $(\mc{X},\mf{S})$, and that the product $F_U \times E_U$ is quasi-isometric to $P_U$.
While $F_U$ and $E_U$ depend on the choice of $p_0$, all choices are uniformly quasi-isometric due to the distance formula. 

In the present work, we shall only work directly with the product region $P_U$, but it is worth noting that if $\diam(CW)$ is uniformly bounded  for all $W \in \mf{S}_U^\perp$, then $E_U$ is a metric space of bounded diameter and $P_U$ has a trivial product structure.  This is one way that we can see the meta-concept that orthogonality is the source of non-hyperbolic behavior in HHSs.

One of the key features of the standard product regions is the existence of a \emph{gate map} $\gate_U \colon \mc{X} \rightarrow P_U$ for each $U \in \mf{S}$.\footnote{Elsewhere in the literature, the gate map is denoted $\gate_{P_U}$ as other subsets, besides the product regions, have gate maps.  As we will only be utilizing the gate maps onto product regions, we have opted to simplify the notation in the present work.} The gate map is a coarsely Lipschitz retract of the entire space onto the product region.  The salient properties of the gate are described in Proposition \ref{prop:gates_to_product_region} below. The map gets it name from Property (4), which says, in order to efficiently travel from a point $x \in \mc{X}$ to a point $p \in P_U$, one must first ``pass through the gate" $\gate_U(x)$.

\begin{prop}[{\cite[Lemma 5.5]{BHS_HHSII}, \cite[Lemma 1.20]{BHS_HHS_Quasiflats}}]\label{prop:gates_to_product_region}
Let $(\mc{X},\mf{S})$ be  a relative HHS. There exists $\mu \geq 1$ such that for each $U \in \mf{S}$, there exists a map $\gate_U \colon \mc{X} \rightarrow P_U$ with the properties:
\begin{enumerate}[(i)]
    \item $\gate_U$ is $(\mu,\mu)$--coarsely Lipschitz. \label{gate:coarsely_lipschitz}
    \item For all $p \in P_U$, $d_\mc{X}(\gate_U(p),p) \leq \mu$. \label{gate:coarsely_idenpontent}
    \item For all $x \in \mc{X}$ and $V \in \mf{S}$, $d_V(\gate_U(x), \rho_V^U) \leq \mu$  if $V \trans U$ or $U \propnest V$ and $d_V(\gate_U(x),x) \leq \mu$ otherwise. \label{gate:Image}
    \item For all $x \in \mc{X}$ and $p \in P_U$, $d_\mc{X}(x,p) \stackrel{\mu,\mu}{\asymp} d_\mc{X}(x, \gate_U(x)) + d_\mc{X}(\gate_U(x),p)$. \label{gate:gate_property}
 
\end{enumerate}

\end{prop}

\subsection{Rank, Orthogonality, and Hyperbolicity}

As discussed in the previous section, the orthogonality relation creates natural, non-trivial product regions in a relative hierarchically hyperbolic space.  Therefore, the orthogonality relation gives rise to a natural notion of rank in relative hierarchically hyperbolic spaces.

\begin{defn}[Hierarchical rank]
Given a relative HHS $(\mc{X},\mf{S})$, the \emph{hierarchical rank of $\mf{S}$}, denoted $\rank(\mf{S})$, is the cardinally of the largest pairwise orthogonal subset of infinite domains in $\mf{S}$.  If $\rank(\mf{S}) = n$, we say $(\mc{X},\mf{S})$ is a \emph{rank $n$ relative HHS}
\end{defn}

Behrstock, Hagen, and Sisto showed that the rank of a relative HHS is always finite \cite[Lemma 2.1]{BHS_HHSII}, and in (non-relative) hierarchically hyperbolic spaces with the bounded domain dichotomy,  that the hierarchical rank agrees with the geometric rank \cite[Theorem 1.14]{BHS_HHS_Quasiflats}.
Further, Behrstock, Hagen, and Sisto show that rank (and hence orthogonality) is the only obstruction to  a hierarchically hyperbolic space being hyperbolic\footnote{This result highlights the importance of the bounded domain dichotomy. Without it, you can have rank 1 HHSs that contain bounded, but arbitrarily large, isometrically embedded Euclidean flats.} .

\begin{thm}[Rank 1 HHSs are hyperbolic. {\cite[Corollary 2.16]{BHS_HHS_Quasiflats}}]\label{thm:hyperbolic_HHSs}
A quasi-geodesic metric space is hyperbolic if and only if it admits a rank 1 HHS structure with the bounded domain dichotomy.
\end{thm}

The goal of Section \ref{sec:rank_1_relative_HHSs} will be to establish a relative version of Theorem \ref{thm:hyperbolic_HHSs}. That is, if $(\mc{X},\mf{S})$ is a rank 1 relative HHS, then $\mc{X}$ is relatively hyperbolic.

\subsection{Distributing the Distance Formula}

Before beginning the main work of this paper, we record the following useful lemma which allows us to ``distribute" the distance formula over a sum when we have a coarse equivalences of distances in each shadow space.

\begin{lem}\label{lem:distributing_distance_formula}
Let $(\mc{X},\mf{S})$ be a relative HHS and  $x_0,x_1, \dots, x_n$ be points in $\mc{X}$. If there exists $C\geq 1$ such that  $ \displaystyle \sum\limits_{i=0}^{n-1} d_W(x_i,x_{i+1}) \stackrel{C,C}{\asymp} d_W(x_0,x_n) $ for all $W \in \mf{S}$, then there exist $L,A \geq 1$, depending only on $C$ and $(\mc{X},\mf{S})$,  such that
\[\sum\limits_{i=0}^{n-1} \left[\frac{1}{A}d_\mc{X}( x_{i}, x_{i+1}) -A \right]  \leq L d_\mc{X} (x_0,x_n) + L.\]
In particular, there exist $K$ depending only on  $C$, $n$, and $(\mc{X},\mf{S})$ such that \[\sum\limits_{i=0}^{n-1} d_\mc{X}( x_{i}, x_{i+1}) \stackrel{K,K}{\asymp} d_\mc{X} (x_0,x_n).  \]
\end{lem}

\begin{proof}

Using the triangle inequality, the second statement follows directly from the first.

  Let $\sigma_0$ be the minimum threshold provided by the distance formula in $(\mc{X},\mf{S})$ and fix $\sigma = 4C^3\sigma_0$.  
  We will show \[\threshold{d_W(x_0,x_n)}{\sigma} \geq  \frac{1}{2C} \sum \limits_{i=1}^{n-1} \threshold{d_W(x_i,x_{i+1})}{C\sigma+C} \tag{$*$} \label{eq:distribute}\] for all $W \in \mf{S}$.
  Once (\ref{eq:distribute}) is established, we will have \[ \sum\limits_{W \in \mf{S}} \threshold{d_{W}(x_0,x_n)}{\sigma} \geq \frac{1}{2C}\sum \limits_{i=0}^{n-1} \sum\limits_{W \in \mf{S}} \threshold{d_{W}(x_i,x_{i+1})}{C\sigma+C}\] and the proposition will follow by applying the distance formula to each side.
  
  First suppose $d_W(x_0,x_n) \leq \sigma$.  
Then $d_W(x_i,x_{i+1})\leq C\sigma + C$ for all $0\leq i\leq n$ and (\ref{eq:distribute}) is satisfied.  
If instead $d_W(x_0,x_n) \geq \sigma$, then \[\sum \limits_{i=0}^{n-1} d_{W}(x_i,x_{i+1}) \geq 2C^2\] which implies \[ \sum \limits_{i=0}^{n-1} d_{W}(x_i,x_{i+1}) - C^2 \geq \frac{1}{2} \sum \limits_{i=0}^{n-1} d_{W}(x_i,x_{i+1}). \] Thus we establish (\ref{eq:distribute}), and finish the proof, with the following calculation
     \[  \threshold{d_{W}(x_0,x_n)}{\sigma} \geq \frac{1}{C}\sum \limits_{i=0}^{n-1} d_{W}(x_i,x_{i+1}) - C \geq \frac{1}{2C}\sum \limits_{i=0}^{n-1} d_{W}(x_i,x_{i+1})\geq \frac{1}{2C}\sum \limits_{i=0}^{n-1} \threshold{d_{W}(x_i,x_{i+1})}{C\sigma+C}. \]
\end{proof}

\section{Rank 1 Relative HHSs are Relatively Hyperbolic} \label{sec:rank_1_relative_HHSs}

In this section, we show that rank 1 relative HHSs are relatively hyperbolic.  We construct a hyperbolic cusped space by attaching a combinatorial horoball to each $P_U$ for $U \in \mf{S}^{rel}$. To show this cusped space is hyperbolic, we build a  rank 1 hierarchically hyperbolic structure for the cusped space from the rank 1 relative HHS structure on $\mc{X}$. The technique to do so is based on the following consequence to the distance formula (Theorem \ref{thm:distance_formula}).

 \begin{lem}\label{lem:minimal_domains_are_subsets}
Let $(\mc{X},\mf{S})$ be a relative HHS. Suppose $\mf{S}$ has  the bounded domain dichotomy and $U\in\mf{S}$ is $\nest$--minimal. If $\mf{S}_U^\perp$ contain no infinite domains,  then $\pi_U\colon P_U \rightarrow CU$ is a quasi-isometry with constants depending only on $(\mc{X},\mf{S})$.
\end{lem}

\begin{proof}
Let $U\in\mf{S}$ be $\nest$--minimal such that $\mf{S}_U^\perp$ contain no infinite domains. Since $(\mc{X},\mf{S})$ has the bounded domain dichotomy, there exists $B>1$, such that $\diam(CV) \leq B$ for all $V \in \mf{S}^\perp_U$.  Thus by taking $\sigma = \sigma_0 B$ in the distance formula for $P_U$, we have that $\pi_U\colon P_U \rightarrow CU$ is a quasi-isometry with constants depending only on $(\mc{X},\mf{S})$.
\end{proof}

Lemma \ref{lem:minimal_domains_are_subsets} says for a rank 1 relative HHS, the product region $P_U$ is quasi-isometric to the shadow space $CU$ for $U \in \mf{S}^{rel}$. Thus, we can define a rank 1 HHS structure on the cusped space by taking the relative HHS structure for $\mc{X}$ and attaching a combinatorial horoball to each $CU$ for $U \in \mf{S}^{rel}$.

\begin{thm}\label{thm:rank_1_relative_HHS}

If $(\mc{X},\mf{S})$ is a rank 1  relative HHS with the bounded domain dichotomy, then  $\mc{X}$ is hyperbolic relative to $\mc{P} = \{ P_U : U \in \mf{S}^{rel} \}$.
\end{thm}

\begin{proof} Let $E>0$ be the hierarchy constant for $(\mc{X},\mf{S})$. 

Since $P_U$ is a uniform quasi-geodesic space for each $U \in \mf{S}^{rel}$ (Proposition \ref{prop:product_regions_are_HHS}), there exists $\epsilon>0$ such that each $P_U$ has an $\epsilon$--separated net $\Gamma_U$. For each $U\in\mf{S}^{rel}$, let $\mc{H}(P_U)$ be the horoball over $P_U$ based on $\Gamma_U$.  Define $\mc{B} = \cusp(\mc{X},\mc{P})$ with this specific choice of horoballs. 

Since $(\mc{X},\mf{S})$ has the bounded domain dichotomy, if $\diam(CU) < \infty$ for some $U \in\mf{S}^{rel}$, then $\diam(P_U)$ is uniformly bounded by Lemma \ref{lem:minimal_domains_are_subsets}. Attaching horoballs to subsets of uniformly bounded diameter does not change the quasi-isometry type of a space, thus we can assume $E$ is large enough that if $\diam(CU) <\infty$, then $CU$ is $E$--hyperbolic. In particular, we can assume  $\diam(CU) = \infty$ for all $U \in\mf{S}^{rel}$. This implies any two elements of $\mf{S}^{rel}$ are transverse as $\rank(\mf{S}) = 1$.

 We now define a rank 1 HHS structure on $\mc{B}$. We  use the index set $\mf{S}$ and maintain the same nesting, orthogonality, and transversality relations. For each $U \in\mf{S}-\mf{S}^{rel}$ the hyperbolic shadow space is $CU$, and for $U\in\mf{S}^{rel}$ the hyperbolic shadow space is $\mc{H}(CU)$, the horoball over $CU$ based on $\pi_U(\Gamma_U)$. $\mc{H}(CU)$ is hyperbolic by Lemma \ref{lem:horoball_distance}. By Lemma \ref{lem:minimal_domains_are_subsets}, $\pi_U \colon P_U \rightarrow CU$ is a uniform quasi-isometry for each $U \in \mf{S}^{rel}$. This quasi-isometry extends to a uniform quasi-isometry $h_U \colon \mc{H}(P_U)\rightarrow \mc{H}(CU)$ such that $h_U(v,n) = (\pi_U(v),n)$ for all $(v,n) \in \Gamma_U \times \mathbb{N} \subseteq \mc{H}(P_U)$. We  denote the projections in the new HHS structure by $\pi_*^b$ and define them as follows.
 \begin{itemize}
     \item For $U \in\mf{S}^{rel}$ define $\pi_U^b: \mc{B} \rightarrow 2^{\mc{H}(CU)}$ by
        \[\pi_U^b(x) = \begin{cases}  \pi_U(x) & x \in \mc{X}  \\
        h_U(x) & x \in \mc{H}(P_U) -P_U \\
        \pi_U(P_V) & x\in \mc{H}(P_V) - P_V \text{ where } V \in \mf{S}^{rel}-\{U\}.
       \end{cases}\]
        
     \item For $U \in \mf{S}-\mf{S}^{rel}$ define $\pi_U^b:\mc{B} \rightarrow 2^{CU}$ by
     \[\pi_U^b(x) = \begin{cases}\pi_U(x) & x \in \mc{X} \\
     \pi_U(P_V) & x\in \mc{H}(P_V) - P_V \text{ where } V \in \mf{S}^{rel}.\end{cases}\]
 \end{itemize}
 
 Since $(\mc{X},\mf{S})$ is rank 1 and has the bounded domain dichotomy, if $V \in \mf{S}^{rel}$ and $U\in \mf{S}-\{V\}$, then either $V\propnest U$, $U\trans V$, or $\diam(CU)$ is uniformly bounded. Hence, $\pi_U(P_V)$ is uniformly bounded by the definition of the product regions (Definition \ref{defn:product_region}).  This, plus the properties of the original projections $\pi_*$, ensures that $\pi_*^b$ satisfies the projection axioms of a hierarchically hyperbolic space. We now verify the remaining axioms.

     \textbf{Nesting, consistency, complexity, and bounded geodesic image:} Since the transversality and nesting relations are inherited from the relative HHS structure and $CU \subseteq \mc{H}(CU)$, we can use the original relative projections $\rho_*^*$. Thus the nesting and finite complexity axioms are satisfied. Since the elements of $\mf{S}^{rel}$ are $\nest$--minimal, the bounded geodesic image axiom follow automatically from the bounded geodesic image axiom in $(\mc{X},\mf{S})$. The consistency axiom follows from the consistency axiom in $(\mc{X},\mf{S})$ and the definition of the product regions.

      \textbf{Orthogonality, containers, and rank:} The orthogonality relation and containers are directly inherited from the original relative HHS structure. Rank 1 follows from the rank of the original relative HHS.
    
      \textbf{Large links:} Let $x, y\in \mc{B}$. The large links axiom is vacuously true for any $\nest$--minimal domain, so it is sufficient to check the axiom for domains $Q \in \mf{S}-\mf{S}^{rel}$. 
     If $x,y \in \mc{X} \subseteq \mc{B}$, then  the conclusion follows immediately from the large links axiom for $(\mc{X},\mf{S})$ and the fact that $d_{\mc{H}(CU)}(p,q) \leq d_{CU}(p,q)$ for all $U \in \mf{S}^{rel}$ and  $p,q \in CU$. Thus, we can assume at least one of $x$ or $y$ are in $\mc{H}(P_U)$ for some $U \in \mf{S}^{rel}$. 
     If both $x,y\in\mc{H}(P_U)$ for some $U \in \mf{S}^{rel}$, then  for all other domains $V\in \mf{S}$, the distance between $\pi_V^b(x)$ and $\pi_V^b(y)$ is uniformly bounded and the axiom holds with $\mf{L}=\{U\}$. Suppose $x \in \mc{H}(P_U)$ for $U \in \mf{S}^{rel}$ and $y\in \mc{X}$. 
     
     Let $Q \in \mf{S} - \mf{S}^{rel}$, then let $\{V_1,\dots,V_m\} \subseteq \mf{S}_Q$ be the domains provided by the large links axiom in $(\mc{X},\mf{S})$ for $y$ and $\gate_{U}(y)$.
    Let $W \in \mf{S}_Q$ and  $d_W(\cdot,\cdot)$ denote distance in $\mc{H}(CW)$ if $W \in \mf{S}^{rel}$ and distance in $CW$ if $W \not \in \mf{S}^{rel}$. 
    If $W \neq U$, then $  \pi_W^b(\gate_{U}(y))\subseteq \pi_W^b(x) $ and we have $d_W(\pi_W^b(x), \pi_W^b(y)) \leq d_W(\pi_W^b(\gate_U(y)), \pi_W^b(y))  .$ 
    Thus if $W\neq U$ and $E < d_W((\pi_W^b(x), \pi_W^b(y))$, then  $E<d_{CW}(\pi_W(\gate_U(y)),\pi_W(y))$. The large links axiom in $(\mc{X},\mf{S})$, then implies $W \nest V_i$, for some $1\leq i \leq m$.
    Therefore, the large links axiom holds for $(\mc{B},\mf{S})$ using the domains $\mf{L} = \{V_1,\dots,V_m, U\}$ and increasing the hierarchy constant by $1$. 
    
     If instead $x \in \mc{H}(P_U)$ and $y \in \mc{H}(P_V)$ for some distinct $U,V \in \mf{S}^{rel}$, then we can use a similar argument by applying the large links axiom in $(\mc{X},\mf{S})$ to $x'\in\gate_{U}(P_V)$ and $y'\in\gate_{V}(P_U)$.

       \textbf{Uniqueness:} Let $x,y \in \mc{B}$ and $\kappa \geq 0$. Let $K>1$ be larger than  the hierarchy constant $E$ for $(\mc{X},\mf{S})$, the constant $\mu$ from Proposition \ref{prop:gates_to_product_region}, and the diameter of $\pi_W^b(x)$ for all $x \in \mc{B}$ and $W\in\mf{S}$. Further, choose $K$ so that for all $U \in \mf{S}^{rel}$, $h_U \colon \mc{H}(P_U) \rightarrow \mc{H}(CU)$ is a $(K,K)$--quasi-isometry and for all $p,q \in CU$,
    \[d_{\mc{H}(CU)}(p,q) \stackrel{K,K}{\asymp} \log(d_{CU}(p,q)).\] 
     Such a $K$ depends only on $(\mc{X},\mf{S})$ by Lemma \ref{lem:horoball_distance} and the fact that each $P_U$ is a uniform quasi-geodesic space (Proposition \ref{prop:product_regions_are_HHS}). If $x,y \in\mc{H}(P_U) -\mc{X}$ for some $U\in\mf{S}^{rel}$, then the axiom follows immediately from the fact that $d_{\mc{B}}(x,y) \leq d_{\mc{H}(P_U)}(x,y) \leq K d_{\mc{H}(CU)}(x,y) + K,$ so we shall assume this is not the case. Define $x', y' \in \mc{X}$ according to the following table where $U,V \in \mf{S}^{rel}$ and $U \neq V$.
     \begin{center}
         \begin{tabular}{|c|c|c|}\hline
            & $y \in \mc{X}$& $y \in \mc{H}(P_V) -\mc{X}$ \\ \hline
              $x \in \mc{X}$& $x' = x$ and $y' = y$   & $x' = x$ and $y' = \gate_V(x)$  \\ \hline
              $x\in \mc{H}(P_U) - \mc{X}$ & $x' = \gate_U(y)$ and $y'=y$ & $x' \in \gate_U(P_V)$ and $y' \in \gate_V(P_U)$ \\\hline
         \end{tabular}
     \end{center}
     
     In all possible cases, we have that $\pi_W(x') \subseteq \pi_W^b(x)$ for $W \in \mf{S}-\{U\}$ and $\pi_W(y') \subseteq \pi_W^b(y)$ for $W \in \mf{S}-\{V\}$. Additionally,  $\pi_U^b(x')$ and $\pi_U^b(y)$ (resp. $\pi_V^b(y')$ and $\pi_V^b(x)$) are at most $K$ far apart. 
     The uniqueness axiom will be satisfied if either  $x \neq x'$ and $d_{\mc{H}(CU)}(x,x') > \kappa +2K$  or $y \neq y'$ and $d_{\mc{H}(CV)}(y,y')>\kappa+2K$, thus we can can restrict our attention to the case where $d_{\mc{H}(CU)}(x,x') \leq \kappa +2K$ and  $d_{\mc{H}(CV)}(y,y') \leq \kappa +2K$. 
     This assumption implies \[d_\mc{B}(x,x') \leq K\kappa+3K^2 \text{ and } d_\mc{B}(y,y') \leq K\kappa+3K^2.\] 
     
     Since $x',y' \in \mc{X}$, the uniqueness axiom for $(\mc{X},\mf{S})$ provides $\theta = \theta(2e^{3K\kappa+5K^2})$ so that if $d_\mc{X}(x',y')>\theta$, then there exists $W \in \mf{S}$ with $d_{CW}(x',y') > 2e^{3K\kappa+5K^2}$.   By the triangle inequality,
     \[d_\mc{B}(x,y) \leq d_\mc{B}(x',y')+2K\kappa + 6K^2 \leq d_\mc{X}(x',y')+2K\kappa+6K^2. \]
    Thus, if $ \theta + 2K\kappa+6K^2 <d_\mc{B}(x,y)$, then $ \theta< d_\mc{X}(x',y')$ and $d_{CW}(x',y') > 2e^{3K\kappa+5K^2}$ for some $W\in\mf{S}$. If $W \not\in \mf{S}^{rel}$, we are finished as $\pi_W(x')\subseteq \pi_W^b(x) \text{ and } \pi_W(y')\subseteq \pi_W^b(y).$ If $W \in \mf{S}^{rel}$, then we have
     \[\frac{1}{K}\log(d_{CW}(x',y') )-K \leq d_{\mc{H}(CW)} (x',y')   \implies \kappa < d_{\mc{H}(CW)} (x,y),\]
     which fulfills the requirements for the uniqueness axiom.

       \textbf{Partial realization:} Let $\{V_1,\dots,V_n\}$ be pairwise orthogonal elements of $\mf{S}$. If no $V_i$ is in $\mf{S}^{rel}$, then the conclusion follows directly from the partial realization axiom in $(\mc{X},\mf{S})$. Suppose, without loss of generality, $V_1 \in \mf{S}^{rel}$. Since no two elements of $\mf{S}^{rel}$ are orthogonal, this implies $V_i \not\in \mf{S}^{rel}$ for $i \neq 1$. Let $p_1 \in \mc{H}(CV_1)$ and $p_i \in CV_i$ for $2\leq i \leq n$.  We claim that $x = h_{V_1}^{-1}(p_1)$ is a point in $\mc{B}$ satisfying the partial realization axiom.
     
     Since $\diam(CV_1)=\infty$, $\rank(\mf{S})= 1$ plus the bounded domain dichotomy imply the diameter of  $CV_i$ is uniformly bounded for all $2 \leq i \leq n$. This guarantees that $x$ satisfies the first requirement of the partial realization axiom. For the second requirement, we can assume $p_1 =(q,n) \in \pi_{V_1}(\Gamma_{V_1}) \times \mathbb{N} \subseteq \mc{H}(CV_1)$. Let $x'\in \mc{X}$ be the point obtained by applying the realization axiom in $(\mc{X},\mf{S})$ to $\{q,p_2,\dots, p_n \}$. If $x'$ is contained in a regular neighborhood of $P_{V_1}$ in $\mc{X}$, then the second requirement of the realization axiom will be satisfied as  $\pi_W^b(x)$ will be uniformly close to $\pi_W(x')$ for all $W \neq V_1$. To show that $x'$ is in a uniform neighborhood of $P_{V_1}$, we will show that $d_{CW}(x',\gate_{V_1}(x'))$ is uniformly bounded for all $W$ and apply the distance formula in $(\mc{X},\mf{S})$ to obtain a uniform bound between $x'$ and $\gate_{V_1}(x') \in P_{V_1}$.
     
     Let $W \in \mf{S} - \{V_1\}$. If $W \perp V_1$, then $\diam(CW)$, and hence $d_{CW}(x',\gate_{V_1}(x'))$, is uniformly bounded. If $W \trans V_1$ or $V_1 \propnest W$, then $d_{CW}(x',\rho_W^{V_1}) \leq E$, by the partial realization axiom in $(\mc{X},\mf{S})$ and $d_{CW}(\gate_{V_1}(x'),\rho_W^{V_1}) \leq \mu$ where $\mu$ is as in Proposition \ref{prop:gates_to_product_region}. Thus $d_{CW}(x',\gate_{V_1}(x')) \leq 2E + \mu$. So, by the distance formula in $(\mc{X},\mf{S})$, $x'$ is contained in a regular neighborhood of $P_{V_1}$ as desired. \\

Since $\mc{B} = \cusp(\mc{X},\mc{P})$ admits a rank 1 HHS structure, $\cusp(\mc{X},\mc{P})$ is hyperbolic by Theorem \ref{thm:hyperbolic_HHSs}, proving that $\mc{X}$ is hyperbolic relative to $\mc{P}$.
\end{proof}

We conclude this section by recording a new characterization of relatively hyperbolic groups obtained by combining Theorem \ref{thm:rank_1_relative_HHS} and Theorem \ref{thm:Relative_HHS_structure_on_rel_hyp}.

\begin{cor}\label{cor:relative_HHS_characterization_of_relative_hyperbolic_groups}
A finitely generated group is relatively hyperbolic if and only if it admits a rank 1 relative HHG structure.
\end{cor}

\section{Hierarchically Hyperbolic Spaces with Isolated Orthogonality}\label{sec:relatively_hyperbolic_HHSs}

We now prove our main result, that the following condition of isolated orthogonality implies that an HHS with the bounded domain dichotomy is relatively hyperbolic.
\begin{defn}[Isolated orthogonality and the factored space]\label{defn:isolated_orthogonality}
Let $(\mc{X},\mf{S})$ be an HHS and $S$ be the $\nest$--maximal element of  $\mf{S}$.  We say $(\mc{X},\mf{S})$  has \emph{orthogonality isolated by $\iso \subseteq \mf{S} -\{S\}$} if the following are satisfied.

\begin{itemize}
    \item If $V,W \in \mf{S}$ and $W\perp V$, then there exists $U \in \iso$ such that $V,W \nest U$.
    \item If $V \in \mf{S}$ and $V \nest U_1$, $V\nest U_2$ for $U_1,U_2 \in \iso$, then $U_1 = U_2$.
\end{itemize}
\noindent We say $(\mc{X},\mf{S})$ has \emph{isolated orthogonality} if it has  orthogonality isolated by some $\iso \subseteq \mf{S} -\{S\}$.
If $(\mc{X},\mf{S})$ has orthogonality isolated by  $\iso$, then the \emph{factored space with respect to $\iso$} is the space obtained from $\mc{X}$ by adding segments, $e_{p,q}$, of length 1 connecting every distinct pair of points $p,q \in P_U$ where $U \in \iso$.  We denote the factored space by $\conespace$.
\end{defn}

The definition of isolated orthogonality is motivated by the HHS structure described in Theorem \ref{thm:hyperbolic_relative_HHS}, which says a space that is hyperbolic relative to HHSs has orthogonality isolated by the collection of peripheral subsets.

The main goal of this section is proving Theorem \ref{thm:isolated_orthogonality_implies_rank_1_RHHS} below, which builds a rank 1 relative HHS structure for an HHS with isolated orthogonality. The proposed relative HHS structure is essentially the standard relative HHS structure on a relatively hyperbolic space described in Theorem \ref{thm:Relative_HHS_structure_on_rel_hyp} with the product regions for the isolating domains taking the role of the peripheral subsets.

\begin{thm}[Isolated orthogonality  implies rank 1 relative HHS]\label{thm:isolated_orthogonality_implies_rank_1_RHHS}
  Let $(\mc{X},\mf{S})$ be an HHS with the bounded domain dichotomy and orthogonality isolated by $\iso$. The following structure $\mf{R}$ is a rank 1 relative HHS structure on $\mc{X}$.
  \begin{itemize}
      \item The index set is  $\mf{R} = \iso \cup \{R\}$.
      The domain $R$ is the $\nest$--maximal domain and any two elements of $\iso$ are transverse. No pair of elements of $\mf{R}$ are orthogonal.
      \item  For each $U \in \iso$, the (non-hyperbolic)  shadow space is $P_U$ and the projection map $\mc{X} \rightarrow P_U$ is the gate map $\gate_{U}$ in $(\mc{X},\mf{S})$. The shadow space for $R$ is  $\conespace$, the factored space of $\mc{X}$ with respect to $\iso$. The projection $\pi_R\colon \mc{X} \rightarrow \conespace$ is the inclusion map $\mc{X} \to \conespace$.
      \item The relative projections are denoted by $\beta_*^*$. If $U,V \in \iso$, then $\beta_U^V = \gate_U(P_V)$ while $\beta_R^U$ is the subset $P_U$ in $\conespace$.
  \end{itemize}
\end{thm}

After proving Theorem \ref{thm:isolated_orthogonality_implies_rank_1_RHHS} we apply Theorem \ref{thm:rank_1_relative_HHS} to conclude that isolated orthogonality implies relative hyperbolicity as claimed in Theorem \ref{thm:intro _isolated_orthogonality_implies_rel_hyp} of the introduction. 

\begin{thm}\label{thm:isolated_orthogonality_implies_Relative_hyperbolicity}
If $(\mc{X},\mf{S})$ is an HHS with the bounded domain dichotomy and orthogonality isolated by $\iso$, then $\mc{X}$ is hyperbolic relative to $\mc{P} = \{P_U : U \in \iso \}$.
\end{thm}

\begin{proof}
If $(\mc{X},\mf{S})$ has orthogonality isolated by $\iso$, then $(\mc{X},\mf{R})$ is a rank 1 relative HHS by Theorem \ref{thm:isolated_orthogonality_implies_rank_1_RHHS}. Further, $\mf{R}^{rel} = \iso$ and the product regions in $(\mc{X},\mf{R})$ for domains in $\iso$ are coarsely equal to the product region in $(\mc{X},\mf{S})$ for domains in $\iso$. Thus, Theorem \ref{thm:rank_1_relative_HHS} says $\mc{X}$ is hyperbolic relative to $\mc{P} = \{P_U : U \in \iso \}$.
\end{proof}

The proof of Theorem \ref{thm:isolated_orthogonality_implies_rank_1_RHHS} is spread over the next two subsections.  We verify all of the axioms except the large links axiom in Section \ref{subsec:axioms_part_1} and verify the large links axiom in Section \ref{subsec:large_link_axiom}.
For the remainder of the section, $(\mc{X},\mf{S})$ will be an HHS with orthogonality isolated by $\iso$, $\conespace$ will be the factored space of $\mc{X}$ as defined in Definition \ref{defn:isolated_orthogonality}, and $\mf{R}$ will be the proposed relative HHS structure for $\mc{X}$ given in Theorem \ref{thm:isolated_orthogonality_implies_rank_1_RHHS}.

\subsection{The proposed relative HHS structure}\label{subsec:axioms_part_1}
In this subsection, we show that  the proposed relative HHS structure $\mf{R}$ in Theorem \ref{thm:isolated_orthogonality_implies_rank_1_RHHS} satisfies all the axioms of a relative HHS structure except the large links axiom.  We begin by collecting a few facts about the product regions for the isolating domains.

\begin{prop}[Properties of isolated orthogonality]\label{prop:isolated_orthogonality}
Let $(\mc{X},\mf{S})$ be an HHS with orthogonality isolated by $\iso \subseteq \mf{S}$. 
\begin{enumerate}[(i)]
    \item \label{item:completly_transverse}For all distinct $U,V \in \iso$, if $W \in \mf{S}_U$ and $Q\in \mf{S}_V$, then $W \trans Q$. 
    \item \label{item:bounded_gate} There exists $B>0$ such that for each distinct $U,V \in \iso$, $\diam(\gate_U(P_{V})) \leq B$.
    \item \label{item:bounded_intersection} For each $r \geq 0$ and $U,V \in \iso$, we have
     \[\mc{N}_r(P_U) \cap \mc{N}_r(P_{V}) \neq \emptyset \implies \mc{N}_r(P_U) \cap \mc{N}_r(P_{V}) \subseteq \mc{N}_{2\mu(r+1)}\bigl(\gate_U(P_V)\bigr)\] where $\mu$ is the constant from Proposition \ref{prop:gates_to_product_region}. In particular, $\diam(\mc{N}_r(P_U) \cap \mc{N}_r(P_{V}))$ is bounded by $4\mu(r+1) +B$ where $B$ the constant from Item (\ref{item:bounded_gate}). 
\end{enumerate}

\end{prop}

\begin{proof}
Let $E$ be the hierarchy constant for $(\mc{X},\mf{S})$, $\mu$ be the constant from Proposition \ref{prop:gates_to_product_region}, and $U$ and $V$ be distinct domains in $\iso$.

Item (\ref{item:completly_transverse}) follows directly from the definition of isolated orthogonality.

 For Item (\ref{item:bounded_gate}), let $x,y \in P_{V}$. It is sufficient to bound $d_W(\gate_U(x), \gate_U(y))$ uniformly for all $W \in \mf{S}$ as the claim will then follow from the distance formula in $(\mc{X},\mf{S})$. 
 By the definition of isolated orthogonality, for all $W \in \mf{S}$, $W \trans U$, $W \nest U$, or $U \nest W$.
 If $W \trans U$ or $U \propnest W$, then the properties of the gate map (Proposition \ref{prop:gates_to_product_region}) imply \[d_W(\gate_U(x), \gate_U(y)) \leq d_W(\gate_U(x),\rho_W^U) + d_W(\rho_W^U,\gate_U(y)) +E \leq 2\mu + E.\] If $ W \nest U$, then $W \trans V$ by Item (\ref{item:completly_transverse}). By the definition of the product region $P_V$ we have $d_W(x,\rho_W^V) \leq E$ and $d_W(y,\rho_W^V) \leq E$ since $W \trans V$. Since $W \nest U$, we have $d_W(\gate_U(x),x)\leq \mu$ and  $d_W(\gate_U(y),y) \leq \mu$ by the properties of the gate map.
 Combining these facts with the triangle inequality produces
 \begin{align*}
   d_W(\gate_U(x), \gate_U(y))  \leq& d_W(\gate_U(x),x) + d_W(x,\rho_W^V) + d_W(\rho_W^V,y) + d_W(y, \gate_U(y)) +3E\\
   \leq& \mu + E+ E + \mu +3E \\
  \leq& 5E +2\mu.
 \end{align*}
 
 For Item (\ref{item:bounded_intersection}), let $r \geq 0$, $x \in \mc{N}_r(P_U) \cap \mc{N}_r(P_{V})$, and $y \in P_V$ with $d_\mc{X}(x,y) \leq r$. By (\ref{gate:coarsely_lipschitz}) and (\ref{gate:gate_property}) of Proposition \ref{prop:gates_to_product_region},  $d_\mc{X}(\gate_U(x),\gate_U(y) ) \leq \mu r+\mu$ and $d_\mc{X}(x,\gate_U(x)) \leq \mu r +\mu$. The triangle inequality now implies \[d_\mc{X}(x, \gate_U(x)) + d_\mc{X}(\gate_U(x),\gate_U(y)) \leq 2\mu r + 2\mu. \] Since $y\in P_V$, this means $\mc{N}_r(P_U) \cap \mc{N}_r(P_{V})$ is contained in the $2\mu(r+1)$--neighborhood of $\gate_U(P_V)$. The final statement now follows by Item (\ref{item:bounded_gate}).
\end{proof}

Next, we show that the shadow space for the $\nest$--maximal domain of $\mf{R}$ is hyperbolic and interacts nicely with hierarchy paths in $(\mc{X},\mf{S})$.  Both of these results follow immediately from the observation that the factored space $\conespace$ in Definition \ref{defn:isolated_orthogonality} is a special case of a more general construction of Behrstock, Hagen, and Sisto (also called a factored space) introduced in Section 2 of \cite{BHS_HHS_AsDim}.

\begin{prop}[The factored space $\conespace$ ]\label{prop:factored_space}
Let $(\mc{X},\mf{S})$ be an HHS with orthogonality isolated by $\iso$ and $\conespace$ be the factored space of $\mc{X}$ with respect to $\iso$.
\begin{enumerate}[(i)]
    \item \label{item:HHS} If $\mf{U} = \{ W \in \mf{S} : W \nest U \text{ for some } U \in \iso \}$, then $\mf{S} - \mf{U}$ is a rank 1 HHS structure on $\conespace$ with the same shadow spaces, projections, and relations as $\mf{S}$.
    \item \label{item:hyperbolic} $\conespace$ is hyperbolic.
    \item \label{item:hp_in_factored} For all $\lambda \geq 1$, there exists $\lambda'\geq 1$ such that if $\gamma$ is a $\lambda$--hierarchy path in $(\mc{X},\mf{S})$, then the inclusion of $\gamma$ into $\conespace$ is an unparameterized $(\lambda',\lambda')$--quasi-geodesic.
    
\end{enumerate}

\end{prop}

\begin{proof}
Let $\mf{U} = \{ W \in \mf{S} : W \nest U \text{ for some } U \in \iso \}$.

We first prove that if $U \in \iso$ and $W \nest U$, then $P_W$ is  contained in a regular neighborhood of $P_U$.
 Let $x \in P_W$ where $W \in \mf{U}$. Let $U \in \iso$ so that $W \nest U$. We will show $d_\mc{X}(x,\gate_U(x))$ is uniformly bounded by showing $d_V(x,\gate_U(x))$ is uniformly bounded for all $V \in \mf{S}$ and applying the distance formula (Theorem \ref{thm:distance_formula}). 

Let $V$ be any domain in $\mf{S}$. First assume $V \in \mf{S}_U$. In this case $d_V(x,\gate_U(x))\leq \mu$ by Proposition \ref{prop:gates_to_product_region}.(\ref{gate:coarsely_idenpontent}). 
Now consider $V \not \in \mf{S}_U$.  Since $W \nest U$ and $U \in \iso$, the definition of isolated orthogonality requires $\mf{S}_W^\perp \subseteq \mf{S}_U$.
 Thus, $V \not\in \mf{S}_U$ implies, $V \trans W$ or $W \propnest V$. Therefore $d_V(x, \rho_V^W) \leq E$, by the definition of the product region $P_W$. Since $V \not\nest U$ by assumption and $V \not \perp U$ by isolated orthogonality, we have $V \trans W$ if and only if $V \trans U$ and $W \propnest V$ if and only if $W \propnest U$.  Thus, the last clause of the consistency axiom says $d_V(\rho_V^W,\rho_V^U) \leq E$.  Finally, since $U \trans V$ or $ U \propnest V$, we have $d_V(\gate_U(x),\rho_V^U) \leq \mu$ by Proposition \ref{prop:gates_to_product_region}.(\ref{gate:Image}).
 Putting these three inequalities together gives us \[d_V(x, \gate_U(x)) \leq d_V(x,\rho_V^W) + d_V(\rho_V^W,\rho_W^U) + d_V(\rho_V^U, \gate_U(x))  + 2E\leq 3E + 2\mu\] for all $V \not \in \mf{S}_U$. Thus, by taking the threshold of the distance formula to be larger than $3E + 2\mu$, we have a bound on $d_\mc{X}(x,\gate_U(x))$ depending only on $(\mc{X},\mf{S})$. In particular, $P_W$ is contained in the $D$--neighborhood of $P_U$ where $D$ depends only on $(\mc{X},\mf{S})$.

 Because each element of  $\{P_W:W \in \mf{U} \}$ is contained in the $D$--neighborhood of a unique element of $\{P_U : U \in \iso\}$, the factored space $\conespace$ described in Definition \ref{defn:isolated_orthogonality} is quasi-isometric to the factor space of $\mc{X}$ with respect to the collection $\mf{U}$ defined in Definition 2.1 of \cite{BHS_HHS_AsDim}. Proposition 2.4 of \cite{BHS_HHS_AsDim} then shows that $\conespace$ admits an HHS structure with index set $\mf{S} - \mf{U}$ and relations, shadow spaces, and projections identical to those for $\mf{S}$. In particular, no two domains of $\mf{S}- \mf{U}$ are orthogonal by the definition of isolated orthogonality. Thus, $(\conespace,\mf{S}-\mf{U})$ is a rank 1 HHS with the bounded domain dichotomy and  $\conespace$ is hyperbolic by Theorem \ref{thm:hyperbolic_HHSs}.

Since the definition of factored space in Definition \ref{defn:isolated_orthogonality} is a special case of the definition of factored space in Definition 2.1 of \cite{BHS_HHS_AsDim}, Item (\ref{item:hp_in_factored}) is  a special case of Lemma 3.11 of \cite{ABD}.
\end{proof}

Finally, we record a special case of a result from \cite{RST_Quasiconvexity} on how hierarchy paths in $(\mc{X},\mf{S})$ interact with  the product regions for domains in $\iso$.

\begin{prop}[Hierarchy paths and isolated product regions]\label{prop:active_subpath_for_morse}
	Let $(\mc{X},\mf{S})$ be an HHS with the bounded domain dichotomy and orthogonality isolated by $\iso$. For all $\lambda\geq 1$, there exist constants \(\nu, D \geq 1\), so that the following holds for all \(x, y\in \mc{X}\) and $U \in\iso$. If $\gamma \colon [a,b]\to \mc{X}$  is a \(\lambda\)--hierarchy path joining \(x\) and \(y\) and  \(d_{{P_U}}(\gate_U (x), \gate_U(y))>D\), then there is a subpath \(\eta=\gamma\vert_{[a_1,b_1]}\) of \(\gamma\) with the properties:
	\begin{enumerate}[(i)]
		\item \(\eta \subseteq \mc{N}_\nu(P_U)\).
		\item The diameters of \(\gate_U\bigl(\gamma([a,a_1])\bigr)\) and \(\gate_U\bigl(\gamma([b_1,b])\bigr)\) are both bounded by $\nu$. \item The distances $d_\mc{X}(\gamma(a_1),\gate_U(x))$  and $d_\mc{X}(\gamma(b_1),\gate_U(y))$ are both bounded by $\nu$.
	\end{enumerate}
\end{prop}

\begin{proof}
For any $U \in\iso$, the definition of isolated orthogonality ensures that $P_U$ has the orthogonal projection dichotomy described in Theorem \ref{thm:quasiconvexity}. Thus $P_U$ is uniformly strongly quasiconvex for any $U \in\iso$, and the proposition is a special case of Proposition 6.18 of \cite{RST_Quasiconvexity}. Note, the first two conclusion of the proposition are stated in Proposition 6.18 of \cite{RST_Quasiconvexity}, while the third conclusion is explicit in the proof.
\end{proof}

We now verify all of the relative HHS axioms, except the large links axiom, required to prove Theorem \ref{thm:isolated_orthogonality_implies_rank_1_RHHS}.

\begin{prop}\label{prop:axioms_part_1}
Let $(\mc{X},\mf{S})$ be an HHS with the bounded domain dichotomy and orthogonality isolated by $\iso$. The proposed rank 1 relative HHS structure $\mf{R}$  described in Theorem \ref{thm:isolated_orthogonality_implies_rank_1_RHHS} satisfies axioms (\ref{axiom:projections}) -- (\ref{axiom:finite_complexity}) and axioms (\ref{axiom:bounded_geodesic_image}) -- (\ref{axiom:partial_realization}) of a relative hierarchically hyperbolic structure for $\mc{X}$.
\end{prop}

\begin{proof}
 Since HHS structures can be transferred over a quasi-isometry and every quasi-geodesic space is quasi-isometric  to a geodesic space (Remark \ref{rem:HHS_and_QI}),  we shall assume $\mc{X}$ is a geodesic metric space.  This implies $\conespace$, the factored space of $\mc{X}$ with respect to $\iso$, is also a geodesic metric space. Let $E$ be the hierarchy constant for $(\mc{X},\mf{S})$.

Recall the proposed relative HHS structure for $\mc{X}$ from Theorem \ref{thm:isolated_orthogonality_implies_rank_1_RHHS}. The index set is $\mf{R} = \iso \cup \{R\}$.  For each $U \in \iso$, the (non-hyperbolic) shadow space is $P_U$ and the projection map $\mc{X} \rightarrow P_U$ is the gate map $\gate_{U}$. The shadow spaces $P_U$ are uniformly quasi-geodesic spaces instead of geodesic space, but this is acceptable by Lemma \ref{lem:quasi-geodesic_shadow_spaces}. The shadow space for $R$ is $\conespace$, the factored space of $\mc{X}$ with respect to $\iso$. $\conespace$ is hyperbolic by Proposition \ref{prop:factored_space}.  The projection map $\pi_R \colon \mc{X} \rightarrow \conespace$ is given by the inclusion map $i\colon \mc{X} \rightarrow \conespace$. The $\nest$--maximal element of $\mf{R}$ is $R$ and every other pair of domains is transverse.  We denote the relative projections in this structure by $\beta_*^*$. For any $U \in \iso$, we define $\beta_R^U$ to be the inclusion of $P_U$ into $\conespace $. As this is a bounded diameter subset by the construction of $\conespace$, the nesting axiom is satisfied. For $U,V \in \iso$, we define $\beta_U^V = \gate_U(P_V)$.

    \textbf{Projections:} The requirements of the projection axiom are met by the properties of the gate map (Proposition \ref{prop:gates_to_product_region}) and the construction of $\conespace$.
 
    \textbf{Uniqueness:} Let $\kappa >0$, $x,y \in \mc{X}$, and $\mf{U} = \{ W \in \mf{S} : W \nest U \text{ for some } U \in \iso \}$.  We show the contrapositive of the uniqueness axiom.  Assume $d_{\conespace}(x,y) \leq \kappa$ and $d_{P_U}(\gate_U(x),\gate_U(y)) \leq \kappa$ for all $U \in \iso$.  By the uniqueness axioms for $(\mc{X},\mf{S}-\mf{U})$ and $(P_U, \mf{S})$ plus the properties of the gate map,
    there exists $\theta'=\theta'(\kappa)$ such that $d_W(x,y) \leq \theta'$ for all $W \in \mf{S}$. By the distance formula in $(\mc{X},\mf{S})$, there exists $\theta = \theta(\kappa)$ such that $d_\mc{X}(x,y) \leq \theta$. 
     
     \textbf{Orthogonality, containers, and rank:} As there is no orthogonality, these axiom are vacuously satisfied and the rank is 1.
     
     \textbf{Consistency:}  Let $U,V \in\iso$ and $\mu>0$ be as in Proposition \ref{prop:gates_to_product_region}.  By Proposition \ref{prop:isolated_orthogonality}.(\ref{item:bounded_gate}), there exists $B = B(\mc{X},\mf{S})>0$ such that the relative projection  $\beta_U^V = \gate_U(P_V)$ is a  subset of $P_U$ of diameter at most $B$.  
     
     Let $\theta$ be the constant from the uniqueness axiom in $(\mc{X},\mf{S})$ with $\kappa = 10E\mu$. We will show that if $d_{P_U}(\gate_{U}(x),\beta_U^V)> \theta$, then $d_{P_V}(\gate_V(x),\beta^V_U)$ is uniformly bounded.
     
     Let $x \in \mc{X}$ such that $d_{P_U}(\gate_{U}(x),\beta_U^V)> \theta$. By the uniqueness axiom in $(\mc{X},\mf{S})$ there exists $W \in \mf{S}$ such that $ d_W(\gate_U(x), \beta_U^V) \geq 10E\mu$. We know $W$ must be an element of $\mf{S}_U$, since $\diam(P_U) \leq 3E$ for all  $Q \in \mf{S} - (\mf{S}_U \cup \mf{S}_U^{\perp})$ and  isolated orthogonality implies $\mf{S}_U^\perp = \emptyset$.  Therefore,  $d_W(x,\gate_U(x)) \leq \mu$ by Proposition \ref{prop:gates_to_product_region}.(\ref{gate:Image}) and hence $d_W(x,\beta_U^V) \geq 9E\mu$ by the triangle inequality.
     
     Since $\mf{S}_V^{\perp} = \emptyset$ by isolated orthogonality, we can uniformly bound $d_{P_V}(\gate_V(x),\beta^V_U)$, by bounding $d_Q(\gate_V(x),\beta_V^U)$ uniformly for all $Q \in \mf{S}_V$ and then applying the distance formula in $P_V$.
     Since $W \nest U$, we have $W \trans V$  by Proposition \ref{prop:isolated_orthogonality}.(\ref{item:completly_transverse}). This implies $d_W(\rho_W^V, \beta_U^V) \leq 2E + \mu$ as the properties of the gate map and the definition of the product region $P_V$ give us \[d_W(\rho_W^V,\gate_U(p)) \leq d_W( \rho_W^V, p) + d_W(p,\gate_U(p)) +E \leq 2E + \mu\] for any $p \in P_V$.
      By Proposition \ref{prop:isolated_orthogonality}.(\ref{item:completly_transverse}), $Q \trans W$ for each $Q \in \mf{S}_V$. Thus, the last clause of the consistency axiom in $(\mc{X},\mf{S})$ ensures that $d_W(\rho_W^V,\rho_W^Q) \leq E$ for all $Q \in \mf{S}_V$. Combining these facts we have the following for all $Q \in \mf{S}_V$:
      \[d_W(x,\rho_W^Q) \geq d_W(x,\beta_U^V) - d_W(\beta_U^V,\rho_W^V) - d_W(\rho_W^V,\rho_W^Q) - 2E \geq 9E\mu - 5E - \mu \geq 3E\mu. \]
      The consistency axiom in $(\mc{X},\mf{S})$ now requires that $d_Q(x ,\rho_Q^W) \leq E$ and $d_Q(\rho_Q^W,\rho_Q^U) \leq E$ for all $Q \in \mf{S}_V$. Applying Proposition \ref{prop:isolated_orthogonality}.(\ref{item:completly_transverse}) again provides $Q\trans U$  for all $Q \in \mf{S}_V$. The properties of the gate map and the definition of the product region then imply  $d_Q(\rho_Q^U,\beta_V^U) \leq \mu +2E$. Hence, we have $d_Q(x, \beta_V^U) \leq 7E+\mu$ for all $Q\in\mf{S}_V$.  Since $d_Q(x,\gate_V(x))\leq \mu$ for all $Q \in \mf{S}_V$, we have $d_Q(\gate_V(x),\beta_V^U) \leq 8E +2\mu$. The  distance formula in $P_V$ now provides a uniform bound on $d_{P_V}(\gate_{V}(x), \beta_V^U)$, completing the proof of the main inequality in the consistency axiom for $(\mc{X},\mf{R})$.
      
      The final clause of the consistency axiom is vacuously satisfied for $(\mc{X},\mf{R})$ as no two elements of $\iso$ are nested.

     \textbf{Partial realization:} Since there is no orthogonality, we only need to verify the axiom for a single domain $U \in \mf{R}$.  If $U = R$ and $p \in \conespace$, then $x=p$ satisfies the axiom. If $U \in \iso$, and $p \in P_U$, then  $x=p$ again satisfies the axiom  by definition of $\beta_*^*$.

     \textbf{Bounded geodesic image:} Since each of the domains in $\iso$ are $\nest$--minimal, we only need verify to this axiom for the domain $R$.  For a subset $A \subseteq \mc{X}$, let $\mc{N}_\nu(A)$ denote the $\nu$--neighborhood of $A$ in $\mc{X}$ and $\widehat{\mc{N}}_\nu(A)$  denote the $\nu$--neighborhood of $A$ in $\conespace$. Let $\lambda_0$ be the constant from Theorem \ref{thm:hierarchy_paths_exist} so that any pair of points in $(\mc{X},\mf{S})$ can be connected by a $\lambda_0$--hierarchy path.
     
     Let $x,y \in \mc{X}$ and $\alpha$ be a geodesic in $\conespace$ between $\pi_R(x)$ and $\pi_R(y)$. Let $D$ be the constant from Proposition \ref{prop:active_subpath_for_morse} for $\lambda= \lambda_0$.  Suppose $d_{P_U}(\gate_{U}(x),\gate_{U}(y))>D$ and let $\gamma$ be a $\lambda_0$--hierarchy path in $(\mc{X},\mf{S})$ between $x$ and $y$. By Proposition \ref{prop:active_subpath_for_morse}, there exists $\nu\geq 0$ depending only on $(\mc{X},\mf{S})$ such that $\gamma \cap \mc{N}_\nu(P_U) \neq \emptyset$. By Proposition \ref{prop:factored_space}.(\ref{item:hp_in_factored}), $\gamma$ is an unparameterized quasi-geodesic in $\conespace$ with constants depending only on $(\mc{X},\mf{S})$. Since $\conespace$ is hyperbolic, this implies there exists $C = C(\mc{X},\mf{S})$ such that $\gamma \subseteq \widehat{\mc{N}}_C(\alpha)$.  Since $\beta_R^U = P_U \subseteq \conespace$ we then have $\alpha \cap \widehat{\mc{N}}_{\nu +C}(\beta_R^U) \neq \emptyset$.
\end{proof}

\subsection{The large links axiom}\label{subsec:large_link_axiom}

We now finish the proof of Theorem \ref{thm:isolated_orthogonality_implies_rank_1_RHHS}, by verifying that the proposed relative HHS  structure $\mf{R}$ satisfies the large links axiom. Since the large links axiom is vacuously true for any $\nest$--minimal domain, we only need to verify the  axiom for the $\nest$--maximal domain of $\mf{R}$. Thus, the axiom  requires us to analyses the following subset of the isolating domains.

\begin{defn}
  Let $(\mc{X},\mf{S})$ be an HHS with orthogonality isolated by $\iso$. For $x,y \in \mc{X}$ define the set
  $ \mf{L}_\tau(x,y) =  \{U \in \iso : d_{P_U}\bigl(\gate_U(x) ,\gate_U(y)\bigr) > \tau\}.$
\end{defn}

Since the shadow space for the $\nest$--maximal element of $\mf{R}$ is the factored space $\conespace$, verifying the large links axiom for the proposed relative HHS structure in Theorem \ref{thm:isolated_orthogonality_implies_rank_1_RHHS} is equivalent to showing that, for sufficiently large $\tau$, the cardinality of $\mf{L}_\tau(x,y)$ is bounded above by a uniform linear function of $d_{\conespace}(x,y)$. We begin by showing $\mf{L}_\tau(x,y)$ contains a finite number of elements that can be linearly ordered along a hierarchy path from $x$ to $y$.

\begin{lem}[Ordering of $\mf{L}_\tau$]\label{lem:linear_ordering_on_L}
Let $(\mc{X},\mf{S})$ be an HHS with orthogonality isolated by $\iso$ and $\lambda_0$ be the constant such that any pair of points in $\mc{X}$ can be joined by a $\lambda_0$--hierarchy path. There exist $\tau_0$ so that for all $\tau\geq \tau_0$ and $x,y \in\mc{X}$, $\mf{L}_\tau(x,y)$ contains a finite number of elements. Further, there exist $K,\nu\geq 1$ so that for any $\lambda_0$--hierarchy path $\gamma \colon [a,b] \to \mc{X}$ connecting $x$ and $y$, the elements of $\mf{L}_\tau(x,y)$ can be enumerated, $U_1,\dots,U_m$, to satisfy the following.

\begin{enumerate}[(i)]
    \item  There exist $a<t_1<\dots <t_m <b$ such that $d_\mc{X}(\gamma(t_i),P_{U_i}) \leq \nu$ for each $i \in \{1,\dots,m\}$ and $|t_i - t_{i+1}| \geq \frac{\tau}{2\lambda_0}$ for all $i \in \{1,\dots,m-1\}$. \label{item:ordering_of_L_1}
    \item For all $1\leq i,j\leq m$, $|i-j| \leq K d_\mc{X}(P_{U_i}, P_{U_j}) + K$. \label{item:ordering_of_L_2}
    \item  If $W \in \iso - \mf{L}_\tau(x,y)$, then $d_\mc{X}(\gate_W(P_{U_i}),\gate_W(P_{U_j})) \leq K$ for all $1\leq i,j\leq m$. \label{item:ordering_of_L_3}
    
\end{enumerate}
\end{lem}

\begin{proof}
Let $\mu \geq 0$ be as in Proposition \ref{prop:gates_to_product_region},  $B\geq 1$ be the constant from Proposition \ref{prop:isolated_orthogonality}.(\ref{item:bounded_gate}), and $D,\nu$ be the constants from Proposition \ref{prop:active_subpath_for_morse} with $\lambda = \lambda_0$. Let $\gamma \colon [a,b] \to \mc{X}$ be a $\lambda_0$--hierarchy path in $(\mc{X},\mf{S})$ connecting $x$ and $y$. Let $\tau \geq 75DB\mu\nu \lambda_0^2$.

     \textbf{Finiteness and Item (\ref{item:ordering_of_L_1}):}
     Since $\tau> D$,  Proposition \ref{prop:active_subpath_for_morse} ensures that for each $U \in \mf{L}_\tau(x,y)$, there exists $a_U,b_U \in [a,b]$ such that $\gamma\vert_{[a_U,b_U]}$ is contained in the $\nu$--neighborhood of $P_U$ and both $d_\mc{X}(\gate_U(x), \gamma(a_U))$ and $d_{\mc{X}}(\gate_U(y),\gamma(b_U))$ are bounded by $\nu$.  Let $ \gamma_U = \gamma\vert_{[a_U,b_U]}$ for each $U \in \mf{L}_\tau(x,y)$. Since $d_\mc{X}(\gate_U(x),\gate_U(y)) \geq \tau$, we have $d_\mc{X}(\gamma(a_U), \gamma(b_U)) \geq \tau - 2\nu$, and hence \[|a_U -b_U| \geq \frac{\tau-2\nu - \lambda_0}{\lambda_0} \geq 70B\mu \nu\lambda_0.\]  On the other hand, the bound on the intersection between the $\nu$--neighborhoods of product regions for elements of $\iso$ (Proposition \ref{prop:isolated_orthogonality}.(\ref{item:bounded_intersection})), implies $\diam(\gamma_U \cap \gamma_V) \leq 4\mu(\nu+1) + B \leq 6B \mu \nu$ for any distinct $U,V \in\mf{L}_\tau(x,y)$. Thus, \[\diam([a_U,b_U] \cap [a_V,b_V]) \leq 6B\mu\nu\lambda_0 +\lambda_0 \leq 7B \mu \nu \lambda_0\] for any distinct $U,V \in\mf{L}_\tau(x,y)$.
     
     Since $|a-b| <\infty$ and every interval in $\{[a_U,b_U] : U \in \mc{L}_\tau(x,y)\}$ is 10 times longer than the length of the intersection of any two intervals in  $\{[a_U,b_U] : U \in \mc{L}_\tau(x,y)\}$, we must have that  $\{[a_U,b_U] : U \in \mc{L}_\tau(x,y)\}$ contains a finite number of elements. However, this implies $\mf{L}_\tau(x,y)$ can contain only a finite number of elements as $[a_U,b_U] = [a_V,b_V]$ if and only if $U =V$. 
     Further, the elements of $\mf{L}_\tau(x,y)$ can be enumerated $\{U_1,\dots, U_m\}$ where  the following hold
     \begin{itemize}
         \item $a\leq a_{U_1} < a_{U_2}< \dots <a_{U_m} < b$;
         \item $a<b_{U_1} <  b_{U_2} < \dots < b_{U_m} \leq b$;
         \item $\gamma_{U_i} \cap \gamma_{U_j} = \emptyset$ whenever $|i-j| \geq 2$.
     \end{itemize}
     This means the subsegments $\gamma_{U_1}, \gamma_{U_2}, \dots, \gamma_{U_m}$, and hence the corresponding product regions, are ordered along $\gamma$ as shown in Figure \ref{fig:ordering_large_links}.

     	 \begin{figure}[ht]
     		\centering
     		\def\svgscale{.7}
     		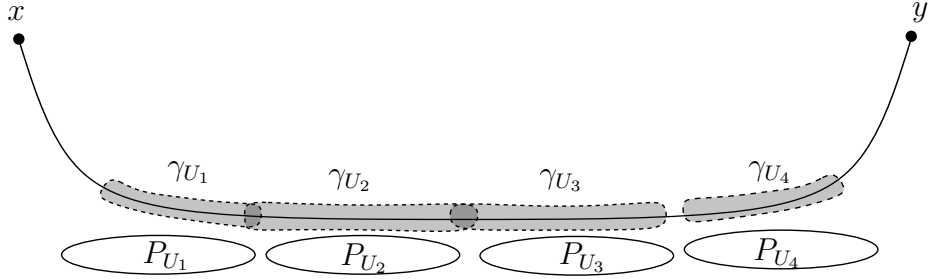
     		\caption{The ordering of the product regions $P_{U_1},\dots,P_{U_m}$ along the hierarchy path $\gamma$.}
     		\label{fig:ordering_large_links}
     	\end{figure}

     Henceforth, let $a_i = a_{U_i}$ and $b_i = b_{U_i}$. For each $i \in \{1,\dots, m\}$, define $t_i \in [a_{i},b_{i}]$ to be the midpoint of $[a_{i},b_{i}]$. Since  $|a_{i} - b_{i} | \geq \frac{\tau -2 \nu - \lambda_0}{\lambda_0} \geq \frac{\tau}{2\lambda_0}$ for each $i \in \{1,\dots,m\}$, we also have $|t_{i} - t_{i+1}| \geq \frac{\tau}{2\lambda_0}$ for each $i \in \{1,\dots, m-1\}$.

     \textbf{Item (\ref{item:ordering_of_L_2}):} Let $U_i,U_j \in \mf{L}_\tau(x,y)$. Without loss of generality, let $i <j$. By possibly increasing our final constant $K$, we can assume $|i-j| \geq 4$. By Lemma 1.20 of \cite{BHS_HHS_Quasiflats}, there exist $K_1 = K_1(\mc{X},\mf{S})\geq 1$, such that \[ d_\mc{X}(P_{U_i},P_{U_j}) \stackrel{K_1,K_1}{\asymp} d_\mc{X} \bigl(\gate_{U_i}(P_{U_j}),\gate_{U_j}(P_{U_i}) \bigr). \tag{$*$} \label{eq:I} \]  
    
    For $\ell \in \{1,\dots,m\}$, let $[a_\ell,b_\ell] \subseteq [a,b]$ be as in the proof of Item (\ref{item:ordering_of_L_1}) above.
    By Proposition \ref{prop:active_subpath_for_morse} and the coarse Lipschitzness of the gate map, there exist $p_i \in \gate_{U_i}(P_{U_j})$ and $p_j \in \gate_{U_i}(P_{U_j})$ such that $p_i$ and $p_j$ are within $3\mu\nu $ of $\gamma(b_i)$ and $\gamma(a_j)$ respectively. Since the diameter of $\gate_{U_i}(P_{U_j})$ and  $\gate_{U_j}(P_{U_i})$ are bounded by $B$, there exist $K_2 = K_2(\mc{X},\mf{S}) \geq 1$ such that
    \[ d_\mc{X}(p_i,p_j) \stackrel{K_2,K_2}{\asymp} d_\mc{X} \bigl(\gate_{U_i}(P_{U_j}),\gate_{U_j}(P_{U_i})\bigr). \tag{$**$} \label{eq:II} \] 
    
    For $\ell \in \{1,\dots,m\}$, let $t_\ell$ be the midpoint of $[a_\ell,b_\ell]$ as in Item (\ref{item:ordering_of_L_1}).  Since $\gamma$ is a $(\lambda_0,\lambda_0)$--quasi-geodesic,  we have $|b_i-a_j| \leq \lambda_0 \cdot d_\mc{X}(p_i,p_j) + 6 \mu \nu \lambda_0 + \lambda_0$. Further, $|t_{i+1} -t_{j-1}| < |b_i-a_j|$. Since $|t_\ell- t_{\ell+1}| \geq \frac{\tau}{2\lambda_0} > 4$ for all $\ell \in \{1,\dots,m-1\}$ and $i < j+1$, we have $|i-j| = |(i+1) - (j-1)| +2 \leq |t_{i+1}- t_{j-1}|$. Combining these results produces
    \[ |i-j| \leq \lambda_0  \cdot d_\mc{X}(p_i,p_j)+6 \mu \nu \lambda_0 + \lambda_0. \tag{$***$} \label{eq:III}\]

The claim follows by combining (\ref{eq:I})--(\ref{eq:III}).
 
    \textbf{Item (\ref{item:ordering_of_L_3}):} Let $W \in \iso$ and suppose $d_\mc{X}(\gate_W(P_{U_i}),\gate_W(P_{U_j})) > K$ for $K$ to be determined later. Assume $i<j$ and let $x' = \gamma(t_i)$ and $y' = \gamma(t_j)$ where $t_i, t_j \in [a,b]$ are as in Item (\ref{item:ordering_of_L_1}). Since $x' \in \mc{N}_\nu(P_{U_i})$ and $y' \in \mc{N}_\nu(P_j)$, we have $d_\mc{X}(\gate_W(x'),\gate_W(y')) \geq K - 4\mu\nu$.  We will use Lemma \ref{lem:distributing_distance_formula} to show that, up to an additive and multiplicative error, $d_\mc{X}(\gate_W(x'),\gate_W(y'))$  is a lower bound for $d_\mc{X}(\gate_W(x),\gate_W(y))$. To apply Lemma \ref{lem:distributing_distance_formula}, we need to find $C = C(\mc{X},\mf{S})$ such that 
    \[d_V\bigl(\gate_W(x),\gate_W(y)\bigr) \stackrel{C,C}{\asymp} d_V\bigl(\gate_W(x),\gate_W(x')\bigr) + d_V\bigl(\gate_W(x'),\gate_W(y')\bigr) + d_V\bigl(\gate_W(y'),\gate_W(y)\bigr) \tag{$\star$} \label{eq:distribute_in_L}\] for all $V \in \mf{S}$. Now, for all $V \in \mf{S} - (\mf{S}_W \cup \mf{S}_W^\perp)$,  we have $\diam(\pi_V(P_W)) \leq 3E$. By taking $C \geq 9E$, we can therefore only consider $V \in \mf{S}_W \cup \mf{S}_W^\perp$. However, because $W \in \iso$, isolated orthogonality implies $\mf{S}_W^{\perp} = \emptyset$. Thus, we only have to consider $V \in \mf{S}_W$.

      Since $\gamma$ is a $\lambda_0$--hierarchy path in $(\mc{X},\mf{S})$, there exist $C_1 = C_1(\mc{X},\mf{S}) \geq 1$ such that for each $V \in \mf{S}_W$ \[d_V(x,y) \stackrel{C_1,C_1}{\asymp} d_V(x,x') + d_V(x',y') + d_V(y',y). \]
    Since  $d_V(\gate_W(z),z) \leq \mu$ for any $V \in \mf{S}_W$ and $z \in \mc{X}$, the above implies that there exists  $C_2 = C_2(\mc{X},\mf{S})>1$ such that (\ref{eq:distribute_in_L}) holds for all $V \in\mf{S}_W$ with $C = C_2$. Taking $C = \max\{C_2,9E\}$ therefore implies (\ref{eq:distribute_in_L}) holds for all $V \in \mf{S}$. We can now apply Lemma \ref{lem:distributing_distance_formula} to obtain  $K_0 = K_0(\mc{X},\mf{S})>1$ such that  
    \[d_\mc{X}(\gate_W(x),\gate_W(y)) \stackrel{K_0,K_0}{\asymp} d_\mc{X}(\gate_W(x),\gate_W(x')) + d_\mc{X}(\gate_W(x'),\gate_W(y')) + d_\mc{X}(\gate_W(y'),\gate_W(y)).\] 
     Thus, if $K > K_0(\tau + K_0) + 4 \mu\nu $, then $d_\mc{X}(\gate_W(x),\gate_W(y)) >\tau$ and $W \in \mf{L}_\tau(x,y)$.
    \end{proof}

The next proposition establishes that if you travel along the product regions for domains in $\mf{L}_\tau(x,y)$ in the order given in Lemma \ref{lem:linear_ordering_on_L}, then you will eventually make measurable forward progress in the factored space $\conespace$ as well.

\begin{prop}[Forward progress in $\conespace$]\label{prop:forward_progress_in_CR}
Let $(\mc{X},\mf{S})$ be an HHS with orthogonality isolated by $\iso$ and $\tau_0$ be the constant from Lemma \ref{lem:linear_ordering_on_L}. Let $\conespace$ be the factored space of $\mc{X}$ with respect to $\iso$.  For each $\tau \geq \tau_0$ and $r \geq 1$, there exist $M\geq 1$ such that for all $x,y\in\mc{X}$ if  the elements of $\mf{L}_\tau(x,y)= \{U_1,\dots,U_m\}$ are enumerated as in Lemma \ref{lem:linear_ordering_on_L}, then  \[d_{\conespace}(P_{U_i},P_{U_j}) \leq r \implies |i-j| \leq M.\]
\end{prop}

\begin{proof}
Assume $d_{\conespace}(P_{U_i},P_{U_j}) \leq r$.  Let $p_i \in P_{U_i}$ and $ p_j \in P_{U_j}$ with $d_{\conespace}(p_i,p_j) \leq r$ and let $\alpha$ be a geodesic in $\conespace$ connecting $p_i$ and $p_j$. By replacing $\mc{X}$ with its approximation graph, we can assume $\mc{X}$ is a metric graph. Thus, the construction of $\conespace$ ensures that $\alpha$ can be decomposed into an alternating concatenation, $b_0 * e_1 * b_1 \dots e_n * b_n$, of geodesics in $\conespace$ such that each  $b_k$ is a geodesics in $\mc{X}$ and each $e_k$ is an edge of length 1 joining two elements of a product region for a domain in $\iso$. We allow for any number of the $b_k$ to have length zero; see Figure \ref{fig:decomposing_alpha}. 
	 \begin{figure}[ht]
	\centering
	\def\svgscale{.7}
	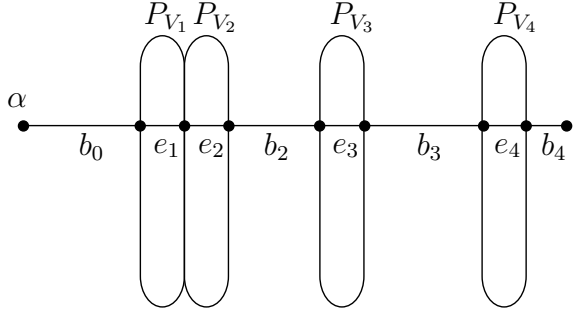
	\caption{The $\conespace$ geodesic $\alpha$ is decomposed into geodesic segments in $\mc{X}$, $b_i$,  and length 1 edges, $e_i$, connecting points in the product region $P_{V_i}$. The segment $b_1$ is not shown as it has length zero.}
	\label{fig:decomposing_alpha}
\end{figure}
As $d_{\conespace}(p_i,p_j) \leq r$, we have $n \leq r$. For each $1\leq k \leq n$, let $P_{V_k}$ be the product region such that $e_k$ connects two points in $P_{V_k}$. Let $V_0 = U_i$ and $V_{n+1} = U_j$.  As the length of each $b_k$ in both $\mc{X}$ and $\conespace$ is at most $r$, we have $d_\mc{X}(P_{V_k},P_{V_{k+1}}) \leq r$ for all $0 \leq k \leq n$.  For $0\leq k,\ell \leq n+1 $ with $k \neq \ell$, let $g_k^\ell \in \gate_{V_k}(P_{V_{\ell}})$.  We first prove a special case of the proposition

\textbf{Claim.} If $V_k \not \in \mf{L}_\tau(x,y)$ for all $1\leq k \leq n$, then for each $r\geq 1$ there exist $M_1\geq 1$ such that $d_{\conespace}(P_{U_i},P_{U_j}) \leq r \implies |i-j| \leq M_1.$
\begin{proof}
By Lemma \ref{lem:linear_ordering_on_L}.(\ref{item:ordering_of_L_2}) it is sufficient to prove $d_\mc{X}(p_i,p_j)$ is bounded above by a quantity depending on $r$, $\tau$, and $(\mc{X},\mf{S})$.
By the triangle inequality \[d_\mc{X}(p_i,p_j) \leq d_\mc{X}(p_i,g_1^0) + \sum \limits_{k=1}^n [d_\mc{X}(g_k^{k-1},g_{k}^{k+1})+ d_\mc{X}(g_k^{k+1},g_{k+1}^k)] + d_\mc{X}(g_{n+1}^{n},p_j).\]

Since $d_\mc{X}(P_{V_k}, P_{V_{k+1}}) \leq r$, Proposition \ref{prop:isolated_orthogonality} implies $\mc{N}_r(P_k) \cap \mc{N}_r(P_{k+1})$ is contained in the $(2\mu(r+1)+B)$--neighborhood of both $g_k^{k+1}$ and $g_{k+1}^k$, where $\mu$ and $B$ depend only on $(\mc{X},\mf{S})$.  Thus $d_\mc{X}(g_k^{k+1},g_{k+1}^k) \leq 4\mu(r+1) + 2B$  all $1\leq k \leq n-1$. Similarly, $d_\mc{X}(p_i,g_1^0)$ and $d_\mc{X}(g_{n+1}^{n},p_j)$ are both bounded by $2\mu(r+1) + B$ as $p_i \in \mc{N}_r(P_{V_0}) \cap \mc{N}_r(P_{V_1})$ and $p_j \in \mc{N}_r(P_{V_{n}}) \cap \mc{N}_r(P_{V_{n+1}})$. As $n\leq r$, the claim will now follow if we can uniformly bound $d_\mc{X}(g_k^{k-1},g_{k}^{k+1})$ for each $1\leq k \leq n$. 

By the triangle inequality

\[ d_\mc{X}(g_k^{k-1}, g_k^{k+1}) \leq \sum\limits_{\ell=0}^{k-2} d_\mc{X}(g_k^{\ell},g_k^{\ell+1})  + \sum\limits_{\ell=k+1}^{n} d_\mc{X}(g_k^{\ell},g_k^{\ell+1}) + d_\mc{X}(g_k^0, g_k^{n+1}).\]

\noindent Since $d_\mc{X}(P_{V_\ell},P_{V_{\ell+1}}) \leq r$ for $1 \leq \ell \leq n$, the coarse Lipschitzness of the gate map and Proposition \ref{prop:isolated_orthogonality}.(\ref{item:bounded_gate}) imply $d_\mc{X}(g^{\ell}_k,g^{\ell+1}_k) \leq \mu r + \mu +2B$ for $0 \leq \ell \leq n$ with $\ell \neq k-1$ or $k$. Since $V_k \not \in \mf{L}_\tau(x,y)$ while  $g_k^0 \in \gate_{V_k}(P_{U_i})$ and $g_k^{n+1} \in \gate_{V_k}(P_{U_j})$,  Lemma \ref{lem:linear_ordering_on_L} provides $K = K(\mc{X},\mf{S},\tau)$ such that $ d_\mc{X}(g_k^0, g_k^{n+1}) \leq K$. Thus $d_\mc{X}(g_k^{k-1},g_{k}^{k+1})$ is uniformly bounded for each $1\leq k \leq n$ and the claim is shown.
\end{proof}

To finish the proof of the proposition, let $0=k_0 < k_1 < \dots < k_s = n+1$ be the indices of the $V_k$ that are elements of $\mf{L}_\tau(x,y)$.  Let $\{i_0,i_1,\dots,i_s\}$ be the indices of the elements of $\mf{L}_\tau(x,y)$ such that $U_{i_\ell} = V_{k_\ell}$. If $k_{\ell-1}< k < k_{\ell}$ for some $1\leq \ell \leq s$, then $V_k \not \in  \mf{L}_\tau(x,y)$. Thus, the above claim  implies $|i_{\ell-1}-i_\ell| \leq M_1$ where $M_1$ depends on $(\mc{X},\mf{S})$, $r$, and $\tau$. Since $s \leq r+2$, this implies $|i-j| \leq M$ where $M$ depends only on $(\mc{X},\mf{S})$, $r$, and $\tau$.
\end{proof}
 
 We now verify that the relative HHS  structure $\mf{R}$ described in Theorem \ref{thm:isolated_orthogonality_implies_rank_1_RHHS} satisfies the large links axiom. This finishes the proof of Theorem \ref{thm:isolated_orthogonality_implies_rank_1_RHHS} that was started in Proposition \ref{prop:axioms_part_1}.
 
\begin{prop}\label{prop:large_links}
Let $(\mc{X},\mf{S})$ be an HHS with the bounded domain dichotomy and orthogonality isolated by $\iso$. The proposed relative HHS structure $\mf{R}$ from Theorem \ref{thm:isolated_orthogonality_implies_rank_1_RHHS} satisfies the large links axiom.
\end{prop}

\begin{proof}
Recall $\mf{R} = \iso \cup \{R\}$ and that $R$ is the nest maximal domain with all elements of $\iso$ being transverse. As the large links axiom is vacuously true for $\nest$--minimal domains, it is sufficient to verify the axiom just for the $\nest$--maximal domain $R$. Recall,  the shadow space for $R$ is $\conespace$, the factored space of $\mc{X}$ with respect to $\iso$. Let $\lambda_0$ be the constant such that every pair of points in $(\mc{X},\mf{S})$ can be joined by a $\lambda_0$--hierarchy path. Let $\gamma$ be a $\lambda_0$--hierarchy path in $(\mc{X},\mf{S})$ connecting $x,y \in \mc{X}$ and $\lambda\geq 1$ be the constant  such that $\gamma$ is an unparameterized $(\lambda,\lambda)$--quasi-geodesic in $\conespace$ (Proposition \ref{prop:factored_space}).

Fix $\tau=\tau_0\lambda^2$ where $\tau_0$ is as in Lemma \ref{lem:linear_ordering_on_L}, and let $\mf{L} = \mf{L}_\tau(x,y)$.  Enumerate the elements of $\mf{L} = \{U_1,\dots,U_m\}$ along $\gamma$ as described in Lemma \ref{lem:linear_ordering_on_L}.  As the elements of $\mf{L}$ are the only elements of $\mf{R} - \{R\}$ where the distance between the gates for $x$ and $y$ is larger than $\tau$, the large links axiom is satisfied if we can show that $m$ is bounded from above by a uniform linear function of $d_{\conespace}(x,y)$. 

Let $x_i = \gamma(t_i)$ where $t_i \in [a,b]$ are as in Lemma \ref{lem:linear_ordering_on_L}.(\ref{item:ordering_of_L_1}). By Proposition \ref{prop:forward_progress_in_CR}, there exist $M >0$ such that if $|i-j| > M$, then $d_{\conespace}(x_i,x_j) > 2 \lambda^2$. Let $1 = i_1< i_2 < \dots < i_n \leq m$ such that $|i_j - i_{j+1}| = M+1$ for $j \in \{1,\dots,n-1\}$. Thus, $d_{\conespace}(x_{i_j},x_{i_{j+1}}) > 2\lambda^2$ for each $j \in \{1,\dots,n-1\}$.

Let $\widehat{\gamma} \colon [0,\ell] \to \conespace$ be the reparameterization of $\gamma$ so that $\widehat{\gamma}$ is a $(\lambda,\lambda)$--quasi-geodesic in $\conespace$. Let $s_0,\dots,s_{n+1} \in [0,\ell]$ such that $\widehat{\gamma}(s_0) = x$, $\widehat{\gamma}(s_{n+1}) = y$, and $\widehat{\gamma}(s_j) = x_{i_j}$ for $j \in \{1,\dots,n\}$. Now $|s_j - s_{j+1}| \geq 1$ for $j \in \{1,\dots,n-1\}$, since $d_{\conespace}(x_{i_j},x_{i_{j+1}}) > 2 \lambda^2$ and  $\widehat{\gamma}$ is a $(\lambda,\lambda)$--quasi-geodesic in $\conespace$. Therefore, we have

\[  n-1 \leq \sum \limits_{j=0}^{n} |s_j - s_{j+1}|  \leq \lambda d_{\conespace}(x,y) + \lambda^2. \]

\noindent Since $m \leq  (M+1)n$, the above inequality implies $m \leq \lambda(M+1)d_{\conespace}(x,y) + \lambda^2(M+1) +1$. This fulfills the requirements of the large link axiom as $\lambda$ and $M$ depend ultimately only on $(\mc{X},\mf{S})$.
\end{proof}

\section{Relative Hyperbolicity in Clean HHGs}\label{sec:clean_hhgs}

In this section, we expand Theorem \ref{thm:isolated_orthogonality_implies_Relative_hyperbolicity} to the characterization of relative hyperbolicity in \emph{clean hierarchically hyperbolic groups} given in Corollary \ref{cor:intro_relatively_hyperbolic_HHGs} of the introduction. Clean HHGs have the following additional hypothesis that is satisfied by every known example of a hierarchically hyperbolic space \cite{ABD,BR_Combination}.

\begin{defn}[Clean Containers]
  A hierarchically hyperbolic structure $\mf{S}$ has \emph{clean containers} if for each $W \in \mf{S}$ and $U \in \mf{S}_W$ with $\mf{S}_W \cap \mf{S}_U^{\perp} \neq \emptyset$, the container for $U$ in $W$ is orthogonal to $U$.
  We say a finitely generated group is a \emph{clean HHG}, if it admits an HHG structure with clean containers.
 \end{defn}

Since every hierarchically hyperbolic group has the bounded domain dichotomy, one direction of Corollary \ref{cor:intro_relatively_hyperbolic_HHGs} is the following immediate corollary of Theorem \ref{thm:isolated_orthogonality_implies_Relative_hyperbolicity} and Theorem \ref{thm:rel_hyp_is_geometric}.

\begin{cor}\label{cor:HHG_with_isolated_orthogonality}
If $(G,\mf{S})$ is an HHG with orthogonality isolated by $\iso \subseteq \mf{S}$, then $G$ is hyperbolic relative to some subgroups $H_1,\dots,H_n$ where each $H_i$ is contained in a bounded neighborhood of some $P_U$ for $U \in \iso$.
\end{cor}

The starting point for a converse of Corollary \ref{cor:HHG_with_isolated_orthogonality} is the observation that the HHG structure described in Theorem \ref{thm:hyperbolic_relative_HHS} for groups hyperbolic relative to HHGs has isolated orthogonality.
Thus, a converse to Corollary \ref{cor:HHG_with_isolated_orthogonality} will follow if we can show that the peripheral subgroups of a relatively hyperbolic HHG are also HHGs. Since the peripheral subgroups  are each strongly quasiconvex, they do inherit the HHS structure from the ambient group.

\begin{prop}[Consequence of {\cite[Proposition 5.7]{RST_Quasiconvexity}} plus {\cite[Proposition 5.6]{BHS_HHSII}}]\label{prop:peripheral_subgroups_are_HHS}
Let $(G,\mf{S})$ be an  HHG and suppose $G$ is hyperbolic relative to a finite collection of peripheral subgroups $\mc{H}$. Each $ H \in \mc{H}$ inherits an HHS structure from $(G,\mf{S})$. That is, $(H,\mf{S})$ is a hierarchically hyperbolic space where for each $U \in\mf{S}$, the shadow space is $\pi_U(H)$ and the projection $H \rightarrow CU$ is given by the restriction of $\pi_U$ to $H$.
\end{prop}

While the HHS structure inherited  by the peripheral subgroup $H$ is equivariant, it fails to be an HHG structure on $H$ as the action of $H$ on the index set need not be cofinite.  However, when the original group is a clean HHG, we can employ a construction of Abbot, Behrstock, and Durham \cite{ABD} to modify the inherited structure on the subgroups to produces an HHG structure for the peripheral subgroups.

\begin{prop}\label{prop:peripherals_are_clean_HHGs}
Let $(G,\mf{S})$ be a clean  HHG. If $G$ is hyperbolic relative to a finite collection of peripheral subgroups $\mc{H}$, then each $H \in \mc{H}$ is a clean HHG.
\end{prop}

\begin{proof}
Before giving the proof we remind the reader of two relevant facts about the relatively hyperbolic HHG $(G,\mf{S})$.

\begin{enumerate}[{Fact} 1:]
    \item \cite[Theorem 4.1]{Drutu_Sapir_Rel_hyp} For any $H \in \mc{H}$, $g \in G$, and $r\geq 0$, if $\diam\left( \mc{N}_r(H) \cap \mc{N}_r(gH) \right) = \infty$, then $gH = H$.
    \item \cite[Lemma 4.15]{RST_Quasiconvexity} There exists  $C \geq 0$ such that for any $U \in \mf{S}$ and $g \in G$, $d_{Haus}(P_{gU},gP_U)<C$.
    \end{enumerate}

Let $H \in \mc{H}$. By Proposition \ref{prop:peripheral_subgroups_are_HHS}, $(H,\mf{S})$ is also an HHS with clean container as the orthogonality relation is inherited from $(G,\mf{S})$. While $H$ acts on $\mf{S}$ with a $\nest$, $\perp$, and $\trans$, preserving action that satisfies Axiom (\ref{item:HHG_equivariance}) of Definition \ref{defn:hierarchically hyperbolic groups}, $\mf{S}$ need not be an hierarchically hyperbolic {group} structure on $H$ as the action of $H$ on $\mf{S}$ need not be cofinite. Thus we will modify the structure on $H$ to ensure that the action on the index set is cofinite.

Since $H$ is strongly quasiconvex in $G$, there exists $B>0$ large enough that $H$ has the $B$--orthogonal projection dichotomy in $(G,\mf{S})$ and $(G,\mf{S})$ has the $B$--bounded domain dichotomy (Theorem \ref{thm:quasiconvexity}). Define \[\mf{T} = \{ U \in \mf{S} : \diam(\pi_U(H)) > B \text{ and there exists } V \in \mf{S}_U^\perp \text{ with } \diam(CV) >B\}.\] By Axiom (\ref{item:HHG_equivariance}) of Definition \ref{defn:hierarchically hyperbolic groups}, $\mf{T}$ is an $H$--invariant collection of domains in $\mf{S}$. Further, $CU \subseteq \mc{N}_B(\pi_U(H))$  for all $U \in \mf{T}$ as $H$ has the $B$--orthogonal domain dichotomy. Since $(H,\mf{S})$ has clean containers, the proof of Theorem 3.12 of \cite{ABD} shows that $H$ admits an HHS structure with index set $\overline{\mf{T}}=\mf{T} \cup \{T\}$ where $T$ is the $\nest$--maximal domain and all other relations are inherited from $\mf{S}$. In particular, $\overline{\mf{T}}$ has clean containers, and the shadow spaces and projections for domains in $\mf{T}$ are the same as in $\mf{S}$.  We now show the action of $H$ on $\overline{\mf{T}}$ is cofinite.

Let $U \in\mf{T}$. Since $\diam(CU)=\infty$ and there exists $V \in \mf{S}_U^\perp$ with $\diam(CV) = \infty$, the orthogonal projection dichotomy implies $CW \subseteq \mc{N}_B(\pi_W(H))$ for all $W \in \mf{S}_U \cup \mf{S}_U^\perp$.  In particular, there exists $D = D(B,\mf{S})$ such that the product region $P_U$ is contained in the $D$--neighborhood of $H$ in $G$ . Now, let $g \in G$ and suppose $U\in \mf{T}$ and $gU \in \mf{T}$. Both $P_U$ and $P_{gU}$ are then contained in $\mc{N}_D(H)$. This implies $gP_U \subseteq \mc{N}_D(gH)$. Thus, the intersection of some uniform finite neighborhoods of $gH$ and $H$ would have infinite diameter as  $gP_U$ is a uniform Hausdorff distance from $P_{gU}$ (Fact 2). By Fact 1, this implies that $g \in H$, and we have that two elements of $\mf{T}$ are in the same $G$ orbit only if they are in the same $H$ orbit.  Since the action of $G$ on $\mf{S}$ is cofinite, this implies the action of $H$ on $\mf{T}$, and hence $\overline{\mf{T}}$, is also cofinite.
\end{proof}

Combining Theorem \ref{thm:hyperbolic_relative_HHS} and Proposition \ref{prop:peripherals_are_clean_HHGs} we have the following corollary that finishes the proof of Corollary \ref{cor:intro_relatively_hyperbolic_HHGs} from the introduction.

\begin{cor}
Let $G$ be a clean HHG. If $G$ is hyperbolic relative to a finite collection of peripheral subgroups $\mc{H}$, then $G$ admits a clean HHG structure with isolated orthogonality.
\end{cor}

\section{Examples of isolated orthogonality in right-angled Coxeter groups}\label{sec:RACG}
A rich collection of clean hierarchically hyperbolic groups with isolated orthogonality can be found among right-angled Coxeter groups. Relative hyperbolicity among all Coxeter groups has been classified by work of Caprace plus Behrstock, Hagen, Sisto, and Caprace; see \cite{Caprace_rel_hyp_coxeter,Caprace_Erratum_rel_hyp_coxeter,BHS_coxeter_thick_random} for details. We show, that in the case of right-angled Coxeter groups, that Caprace's criteria for relative hyperbolicity will imply the right-angled Coxeter group admits an HHG structure with isolated orthogonality. This allows us to use Theorem \ref{thm:isolated_orthogonality_implies_Relative_hyperbolicity} to produce a new proof that these right-angled Coxeter groups are relatively hyperbolic.

For a finite simplicial graph $\Gamma$, let $V(\Gamma)$ and $E(\Gamma)$ denoted the vertex and edge sets of $\Gamma$ respectively. The \emph{right-angled Coxeter group with defining graph $\Gamma$} is then the group $W_\Gamma$ with presentation
\[W_\Gamma = \langle v \in V(\Gamma) \mid v^2 = 1 \ \forall v \in V(\Gamma), \ vw=wv \ \forall \{v,w\} \in E(\Gamma) \rangle.\]
Given a full subgraph $\Lambda \subseteq \Gamma$, the subgroup  of $W_\Gamma$ generated by $V(\Lambda)$ is isomorphic to the right-angled Coxeter group $W_\Lambda$. Thus, we will abuse notation and use $W_\Lambda$ to denote the subgroup of   $W_\Gamma$ generated by $V(\Lambda)$. We call such subgroups the \emph{graphical subgroups} of $W_\Gamma$.\footnote{These subgroups are often called \emph{parabolic subgroups} in the literature. We have opted for the name graphical subgroups to avoid confusion with the peripheral subgroups of a relatively hyperbolic group, which are often also referred to as parabolic subgroups.} Note, a right-angled Coxeter group $W_\Gamma$ is a finite group if  and only if $\Gamma$ is a complete graph.

Behrstock, Hagen, and Sisto proved that every right-angled Coxeter group is a hierarchically hyperbolic group \cite{BHS_HHSI}. To describe this HHG structure, we need to recall some terminology from graph theory: Given a full subgraph $\Lambda \subseteq \Gamma$, let $\lk(\Lambda)$ denote the link of $\Lambda$---the full subgraph of $\Gamma$ spanned by the vertices that are joined by an edge to \emph{every} vertex of $\Lambda$---and let $\st(\Lambda)$ denote the star of $\Lambda$---the full subgraph spanned by $\Lambda \cup \lk(\Lambda)$. We say two non-empty subgraphs $\Lambda$ and $\Omega$ form a join $\Lambda \bowtie \Omega$, if every vertex of $\Lambda$ is joined by an edge to every vertex of $\Omega$. Note, the right-angled Coxeter group $W_{\Lambda \bowtie \Omega}$ is isomorphic to the direct product $W_\Lambda \times W_\Omega$. 

We summarize the index set and relations for the HHG structure for $W_\Gamma$. We omit any discussion of the hyperbolic spaces and projection maps as they are not relevant to verifying isolated orthogonality. We direct the reader to \cite{BHS_HHSI} for full details on the HHS structure. Given two cosets $g W_\Lambda$  and $h W_\Lambda$ of the same graphical subgroup of $W_\Gamma$, we say $g W_\Lambda$ is \emph{parallel} to $h W_\Lambda$ if $g^{-1} h \in W_{\st(\Lambda)}$. Parallelism between graphical subgroups is an equivalence relation.  The index set $\mf{S}$ for the HHG structure on $W_\Gamma$ is the set of all parallelism classes of graphical subgroups whose defining graph is not a complete graph. Two parallelism classes $[gW_\Lambda],[hW_\Omega] \in \mf{S}$ are orthogonal if  $\Lambda \subseteq \lk(\Omega)$ and there exists $k \in W_\Gamma$ such that $[gW_\Lambda] = [kW_\Lambda]$ and $[hW_\Omega] = [kW_\Omega]$. Similarly, $[gW_\Lambda] \nest [hW_\Omega]$ if $\Lambda \subseteq \Omega$ and there exists $k \in W_\Gamma$ such that $[gW_\Lambda] = [kW_\Lambda]$ and $[hW_\Omega] = [kW_\Omega]$. For any element of $[g W_\Lambda] \in \mf{S}$, the product region $P_{[gW_\Lambda]}$ is within finite Hausdorff distance from the coset $g W_{\st(\Lambda)}$.

We now verify the following theorem that recovers a portion of Theorem A$'$ of \cite{Caprace_Erratum_rel_hyp_coxeter}.

\begin{thm}[Right-angled Coxeter groups with isolated orthogonality]\label{thm:relative_hyperbolicity_of_RACG}
	Let $W_\Gamma$ be a right-angled Coxeter groups and $\mf{S}$ be the HHG structure outlined above. Suppose there exists a collection $\mc{J}$ of proper, full, non-complete subgraphs of $\Gamma$ satisfying the following:
	\begin{enumerate}[(i)]
		\item If two non-complete full subgraphs $\Lambda_1$ and $\Lambda_2$ form a join subgraph, then there exists $\Omega \in \mc{J}$ such that $\Lambda_1 \bowtie \Lambda_2 \subseteq \Omega$. \label{item:RACG_1}
		\item If $\Omega_1, \Omega_2 \in \mc{J}$ and $\Omega_1 \neq \Omega_2$, then $\Omega_1 \cap \Omega_2$ is either empty or a complete graph.\label{item:RACG_2}
		\item If $\Lambda$ is a full subgraph of $\Gamma$ that is not a complete graph and $\Lambda \subseteq \Omega$ for some $\Omega \in \mc{J}$, then $\lk(\Lambda) \subseteq \Omega$. \label{item:RACG_3}
	\end{enumerate} 
	Then, $(W_\Gamma,\mf{S})$ has orthogonality isolated by $\iso = \{[gW_\Omega] : \Omega \in \mc{J} \text{ and } g \in W_\Gamma\}$. In particular, $W_\Gamma$ is hyperbolic relative to the subgroups $\{W_\Omega : \Omega \in \mc{J}\}$
\end{thm}

\begin{proof}
	Since $\mc{J}$ contains proper subgraphs of $\Gamma$, $\iso$ does not contain $ [W_\Gamma]$, the $\nest$--maximal element of $\mf{S}$.

	Let $[gW_\Lambda] ,[hW_\Pi] \in \mf{S}$ with $[gW_\Lambda] \perp [hW_\Pi]$. Thus, $\Lambda \subseteq \lk(\Pi)$ and there exists $k \in W_\Gamma$ such that $[k W_\Lambda] = [g W_\Lambda]$ and $[kW_\Gamma] = [hW_\Gamma]$.  Since $\Lambda \subseteq \lk(\Pi)$, the subgraphs $\Lambda$ and $\Pi$ form a join. Thus, there exists $\Omega \in \mc{J}$ such that $\Lambda \bowtie \Pi \subseteq \Omega$ by Item (\ref{item:RACG_1}). This implies $[gW_\Lambda] = [kW_\Lambda] \nest [k W_\Omega]$ and $[hW_\Pi] = [kW_\Pi] \nest [kW_\Omega]$ satisfying the first part of isolated orthogonality.
	
	Now suppose $[gW_\Lambda] \nest [h_1W_{\Omega_1}]$ and $[gW_\Lambda] \nest [h_2W_{\Omega_2}]$ for $[h_1 W_{\Omega_1}], [h_2 W_{\Omega_2}] \in \iso$. The definition of nesting then requires that $\Lambda \subseteq \Omega_1 \cap \Omega_2$. This implies $\Omega_1 \cap \Omega_2$ is not a complete subgraph since $\Lambda$ is not a complete subgraph and all graphs are full subgraphs. By Item (\ref{item:RACG_2}), this means $\Omega_1 = \Omega_2$. Let $\Omega = \Omega_1 = \Omega_2$
	
	Since $[gW_\Lambda] \nest [h_1W_{\Omega}]$ and $[gW_\Lambda] \nest [h_2W_{\Omega}]$, there exists $k_1,k_2 \in W_\Gamma$ such that $[gW_\Lambda] = [k_1 W_\Lambda] = [k_2 W_\Lambda]$, $[h_1 W_\Omega] = [k_1 W_\Omega]$, and $[h_2 W_\Omega] = [k_2 W_\Omega]$.  Because $[k_1 W_\Lambda] = [k_2 W_\Lambda]$, we have $k_1^{-1}k_2 \in W_{\st(\Lambda)}$. By Item (\ref{item:RACG_3}), $\Lambda \subseteq \Omega \implies \st(\Lambda) \subseteq \Omega$. Thus, $k_1^{-1}k_2 \in W_{\st(\Omega)}$ since $W_{\st(\Lambda)} \subseteq W_{\Omega} \subseteq W_{\st(\Omega)}$. This implies  $[h_1 W_\Omega] = [k_1 W_\Omega] = [k_2 W_\Omega] = [h_2 W_\Omega]$ as required by the second part of isolated orthogonality
	
	Since $(W_\Gamma,\mf{S})$ has orthogonality isolated by $\iso$, Theorem \ref{thm:isolated_orthogonality_implies_Relative_hyperbolicity} says the Cayley graph of   $W_\Gamma$ is a relatively hyperbolic metric space with peripheral subsets $\{P_{[hW_\Omega]} : \Omega \in \mc{J} \text{ and } h \in W_\Gamma \}$. Since the product region $P_{[hW_\Omega]}$ is within finite Hausdorff distance of the coset $hW_{\st(\Omega)}$, Theorem \ref{thm:rel_hyp_is_geometric} says $W_\Gamma$ is hyperbolic relative to the subgroups $\{W_{\st(\Omega)} : \Omega \in \mc{J}\}$. Since no element of $\mc{J}$ is a complete graph, Item (\ref{item:RACG_3}) implies that $\st(\Omega) = \Omega$ for all $\Omega \in \mc{J}$, thus $W_\Gamma$ is hyperbolic relative to the subgroups $\{W_{\Omega} : \Omega \in \mc{J}\}$
\end{proof}

\begin{remark}
	Theorem A$'$ of \cite{Caprace_Erratum_rel_hyp_coxeter} says that the conditions on $\mc{J}$ in Theorem \ref{thm:relative_hyperbolicity_of_RACG} are both necessary and sufficient for $W_\Gamma$ to be hyperbolic relative to the graphical subgroups $\{W_\Omega : \Omega \in \mc{J}\}$. Further, Behrstock, Hagen, Sisto, and Caprace showed that $W_\Gamma$ is relatively hyperbolic if and only if there exists a collection $\mc{J}$ of proper subgraphs of $\Gamma$ satisfying the requirement of Theorem \ref{thm:relative_hyperbolicity_of_RACG} \cite[Appendeix A]{BHS_coxeter_thick_random}. These results allows us to conclude that a right-angled Coxeter group is relatively hyperbolic if and only if the HHG structure described above has isolated orthogonality.
\end{remark}

\section{Graphs Associated to Surfaces}\label{sec:complexes_of_curves}

In this final section, we apply Theorem \ref{thm:isolated_orthogonality_implies_Relative_hyperbolicity} to prove the relative hyperbolicity of certain graphs built from the curves on a surface.  We begin by recalling some standard terminology when working with curves on surfaces.

For the remainder of the section, $S_{g,n}$ will denote a connected, orientable, surface with genus $g$ and $n$ boundary components. It will be immaterial whether or not the boundary components are punctures or curves. The \emph{complexity} of $S=S_{g,n}$ is $\xi(S) = 3g-3+n$.  By a \emph{curve} on $S$ we mean an isotopy class of an essential, non-peripheral, simple closed curve on $S$.  By a \emph{subsurface} of $S$, we mean an isotopy class of an essential, non-peripheral, closed, not necessarily connected, subsurface of $S$. We say two curves and/or subsurfaces are \emph{disjoint}, if their isotopy classes can be realized disjointly. A \emph{multicurve} on $S$ is a collection of distinct, pairwise disjoint curves on $S$.  For two subsurfaces $U$ and $V$, we say $U\subseteq V$ if $U$ and $V$ can be realized such that $U$ is contained in $V$. Given a closed subsurface $U \subseteq S$, let $U^c$ denote the closed subsurface whose interior is isotopic to $S-U$.

 We say two curves on $S$ \emph{intersect minimally} if they can be realized to have the smallest number of intersections for any pair of curves on $S$. If $\xi(S) >1$, then intersecting minimally is the same as disjointness. The \emph{curve graph} of $S$, $CS$, is the graph where the vertices are curves on $S$ and two curves are joined by an edge if and only if they intersect minimally.  If $S$ is a finite disjoint union of connected surfaces, then the curve graph $CS$ is the graphical join of the curve graphs of each of the connected components of $S$.
 
 The curve graph is a special case of a larger collection of graphs associated to surfaces that we study in this section.
 
 \begin{defn}
   A \emph{graph of multicurves} on a surface $S = S_{g,n}$ is a graph whose vertices are multicurves on $S$. If $\mc{G}$ is graph of multicurves on $S$, then we say a connected subsurface $W \subseteq S$ is a \emph{witness} for $\mc{G}$ if the interior of $W$ is not homeomorphic to a $3$--holed sphere and every vertex of $\mc{G}$ is not disjoint from $W$. We denote the set of witnesses by $\Wit(\mc{G})$.
 \end{defn}

Each graph of multicurves is a metric space by declaring each edge to have length 1 and Vokes showed that under natural conditions a graph of multicurves admits an HHS structure utilizing the subsurface projection machinery of Masur and Minsky \cite{MMII}.

\begin{defn}
Let $S=S_{g,n}$ and $\mc{G}$ be a graph of multicurves on $S$. We say $\mc{G}$ is \emph{twist free} if $\Wit(\mc{G})$ contains no annuli. We say $\mc{G}$ is \emph{hierarchical} if $\mc{G}$ is connected; the action of the mapping class group on $S$ induces an action by graph automorphisms on $\mc{G}$; and there exist $R\geq 0$ such that any pair of adjacent vertices of $\mc{G}$ intersect at most $R$ times.
\end{defn}

\begin{thm}[{\cite[Theorem 1.1]{vokes}}]\label{thm:multicurve_graphs_are_HHS}
Let $S=S_{g,n}$. If $\mc{G}$ is a hierarchical and twist free graph of multicurves on $S$, then $\mc{G}$ is a hierarchically hyperbolic space with the following structure.
\begin{itemize}
    \item The index set $\mf{S}$ is the collection of all (not necessarily connected) subsurfaces of $S$ such that  each connected component of $U\in \mf{S}$ is an element of $\Wit(\mc{G})$.
    \item For each $U \in \mf{S}$, the shadow space is the curve graph $CU$ and the projection map is the subsurface projection defined in  \cite{MMII}.
    \item For $U,V \in \mf{S}$, $U \perp V$ if $U$ and $V$ are disjoint, $U \nest V$ if $U \subseteq V$, and $U \trans V$ otherwise. The relative projections are defined by taking the subsurface projection of the boundary of one subsurface onto the other.
\end{itemize}
\end{thm}

We will focus on three specific examples of hierarchical, twist free graphs of multicurve, each of which will be relatively hyperbolic for certain surfaces.

\begin{ex}[The separating curve graph]\label{ex:separating_curve_graph}
The vertices of the separating curve graph,  $\Sep$,  are all the separating curves on $S = S_{g,n}$. Two separating curves are joined by an edge if they are disjoint. The witnesses for $\Sep$ are connected subsurfaces $U \subseteq S$ such that each component of $U^c$ contains no genus and at most $1$ boundary component of $S$. $\Sep$ satisfies the requirements of Theorem \ref{thm:multicurve_graphs_are_HHS} whenever it is non-empty and connected. This occurs when $2g+n \geq 5$ and $S \neq S_{2,1}$. 
\end{ex}

\begin{ex}[The pants graph]\label{ex:pants_graph}
The vertices of the pants graph, $\mc{P}(S)$, are all multicurves that define pants decompositions of $S=S_{g,n}$. Two multicurves $x,y \in \mc{P}(S)$ are joined by an edge if  there exist curves $\alpha \in x$ and $\beta \in y$ such that $(x-\alpha) \cup \beta = y$ and $\alpha$ and $\beta$ intersect minimally on the complexity $1$ component of $S - (x -\alpha)$. The witnesses for $\mc{P}(S)$ are all connected subsurfaces with complexity at least $1$. $\mc{P}(S)$ satisfies the requirements of Theorem \ref{thm:multicurve_graphs_are_HHS} when $\xi(S) \geq 1$.  While Theorem \ref{thm:multicurve_graphs_are_HHS} provided a new proof that $\mc{P}(S)$ is an HHS, this fact was originally deduced from results in \cite{MMI,MMII,Behrstock_Thesis,BKMM_Rigidity,Brock_WeilPetersson}. Brock showed that the pants graph of $S$ is quasi-isometric to the Weil--Petersson metric on the Teichm\"uller space of $S$ \cite{Brock_WeilPetersson},  allowing the above to also be an HHS structure on the Weil--Petersson metric; see Remark \ref{rem:HHS_and_QI}. 
\end{ex}

\begin{ex}[The cut system graph]\label{ex:hatcher_thurston}
For a closed surface $S=S_{g,0}$, the cut system graph, $\cut(S)$, is the $1$--skeleton of the complex studied in \cite{Hatcher_Thurston_complex}. The vertices of $\cut(S)$ are all multicurves $x$ such that $S-x$ contains no genus and multicurves $x,y \in \cut(S)$ are joined by an edge if there exist curves $\alpha \in x$ and $\beta \in y$ such that $(x-\alpha) \cup \beta = y$ and $\alpha$ and $\beta$ intersect once. The witnesses for $\cut(S)$ are all connected subsurfaces containing genus. $\cut(S)$ satisfies the requirements of Theorem \ref{thm:multicurve_graphs_are_HHS} when $g\geq 1$. The cut system graph is also called the \emph{Hatcher--Thurston graph} in the literature.
\end{ex}

By applying Theorem \ref{thm:hyperbolic_HHSs} to Theorem \ref{thm:multicurve_graphs_are_HHS}, Vokes characterizes when a hierarchical, twist free graph of multicurves is hyperbolic.

\begin{cor}[{\cite[Corollary 1.5]{vokes}}]\label{cor:hyperbolic_graphs_of_multicurves}
Let $\mc{G}$ be a hierarchical and twist free graph of multicurves on $S_{g,n}$. $\mc{G}$ is hyperbolic if and only if $\Wit(\mc{G})$ contain no disjoint subsurfaces.
\end{cor}

Since the curve graph of a surface is either infinite diameter or has diameter at most $2$, the HHS structure described in Theorem \ref{thm:multicurve_graphs_are_HHS} will always have the bounded domain dichotomy. Thus, Theorem \ref{thm:isolated_orthogonality_implies_Relative_hyperbolicity} allows us to detect the relative hyperbolicity of such graphs from isolated orthogonality on the index set.

In the next theorem, we apply Theorem \ref{thm:isolated_orthogonality_implies_Relative_hyperbolicity} to show the relative hyperbolicity of the pants graph, separating curve graph, and cut system graph in specific cases. The fact that the pants graph is relatively hyperbolic for surfaces with complexity $3$ is originally a result of Brock and Masur \cite[Theorem 1]{Brock_Masur_WP_Low_complexity}, while the case of the cut system graph of a genus $2$ surface was first shown by Li and Ma \cite[Theorem 1.2]{Li_Ma_Hatcher-Thurston}. Theorem \ref{thm:rel_hyp_complexes_of_curves} provides a new proof of these results and the first proof of the relative hyperbolicity of the separating curve graph of a closed or twice punctured surface.

\begin{thm}\label{thm:rel_hyp_complexes_of_curves}
Let $\mc{G}$ we one of the following graphs.
\begin{enumerate}[(i)]
    \item $\mathrm{Sep}(S_{g,n})$ when $2g+n \geq 6$ and $n=0$ or $2$.
    \item $\mc{P}(S_{g,n})$ when $\xi(S_{g,n}) = 3$.
    \item $\cut(S_2)$.
\end{enumerate}
If $\mf{S}$ is the HHS structure from Theorem \ref{thm:multicurve_graphs_are_HHS} for $\mc{G}$, then $(\mc{G},\mf{S})$ has orthogonality  isolated by $\iso = \{W \sqcup W^c : W \in \Wit(\mc{G}) \text{ with } W^c \in \Wit(\mc{G})\}$.
In particular, $\mc{G}$ is a relatively hyperbolic space where every peripheral subset is quasi-isometric to the product of two infinite diameter curve graphs $CW \times CW^c$ for some $W \sqcup W^c \in \iso$.
\end{thm}

\begin{proof}
We first show that, for each graph, the set of witnesses has the property that if $U$ and $V$ are disjoint witnesses, then $U = V^c$. We call this property \emph{unique disjoint pairs}. 

Let $U$ and $V$ be a pair of disjoint witnesses for $\mathrm{Sep}(S_{g,n})$. Let $W$ be the component of $V^c$ that contains $U$. Then $W$ and $W^c$ are both witnesses for $\mathrm{Sep}(S_{g,n})$ as they are both connected and $U \subseteq W$ and $V \subseteq W^c$. This implies that both $W$ and $W^c$ cannot contain any genus and each contains exactly one boundary component of $S$ if $S = S_{g,2}$; see Figure \ref{fig:Sep_witnesses}. Now if $U \neq W$, then $W -U$ contains a non-annular subsurface $Z$ and $Z$ contains at least two of the boundary curves of $W$. Note, at least one of these boundary curves cannot be the boundary of $S$ as $W$ contains only a single boundary component of $S$.
Now, $Z$ cannot meet the boundary of $W^c$ in  two or more curves, as that would imply $U^c$ contains genus. Thus, $Z$ must meet the boundary of $W^c$ in a single curve and $Z$ must contain the boundary component of $S$ that is contained in $W$. 
However, this implies that a component of $U^c$ will contain both boundary components of $S$, which is also impossible. Thus we must have $W=U$. An identical argument implies $W^c = V$. Therefore, $V = U^c$ as desired.

\begin{figure}[ht]
	\centering
	\def\svgscale{.7}
	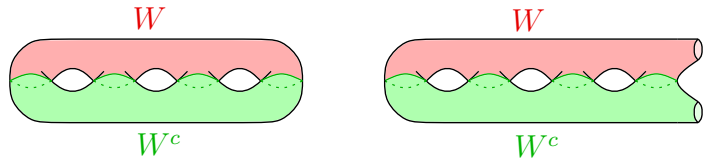
	\caption{The connected disjoint witnesses $W$ and $W^c$ for the separating curve graph when $n = 0$ or $n=2$.}
		\label{fig:Sep_witnesses}
\end{figure}

If $U$ and $V$ are a pair of disjoint witnesses for $\mc{P}(S_{g,n})$ when $\xi(S_{g,n}) = 3$, then $U$ and $V$ must be a pair of complexity 1 subsurfaces that share a single boundary curve as these are the only pairs of disjoint subsurfaces with complexity at least 1. Thus $V = U^c$ as desired. 
  
  Similarly, if $U$ and $V$ are a pair of disjoint witnesses for $\cut(S_2)$, then $U$ and $V$ must be a pair of  1--holed tori  with a common boundary as these are the only pairs of disjoint subsurfaces that both contain genus.  Hence $V = U^c$ as desired.

Let $\mc{G}$ be the separating curve graph, pants graph, or cut system graph for any of the cases listed in Theorem \ref{thm:rel_hyp_complexes_of_curves} and $\mf{S}$ be the HHS structure for $\mc{G}$ from Theorem \ref{thm:multicurve_graphs_are_HHS}. We will show that $(\mc{G},\mf{S})$ has orthogonality isolated by $\iso = \{W \sqcup W^c : W,W^c \in \Wit(\mc{G})\}$.

Let $U,V \in \mf{S}$ such that $U \perp V$. By definition of $\mf{S}$, there exist connected subsurfaces $U'\subseteq U$ and $V' \subseteq V$ with $U'$ disjoint from $V'$. By unique disjoint pairs, $V' = (U')^c$. This implies that $U'=U$, $V'=V$, and $V = U^c$. Thus $U,V \nest U \sqcup U^c$ and $U \sqcup U^c \in \iso$. 

Now suppose $U \in \mf{S}$ and there exists $W_1,W_2 \in \Wit(\mc{G})$ such that $W_1^c, W_2^c \in \Wit(\mc{G})$ and $U$ is nested into both $W_1 \sqcup W_1^c$ and $W_2 \sqcup W_2^c$. By swapping $W_i$ with $W_i^c$ if needed, we can assume there exists a connected subsurface $U' \subseteq U$ such that $U' \subseteq W_1$, $U'\subseteq W_2$, and $U' \in \Wit(\mc{G})$. Unique disjoint pairs therefore requires $W_1^c = W_2^c = (U')^c$ and $U'=W_1 = W_2$. 

The previous two paragraphs demonstrate that $(\mc{G},\mf{S})$ has orthogonality isolated by $\iso$. Therefore, $\mc{G}$ is hyperbolic relative to $\{ P_{W \sqcup W^c} : W \sqcup W^c \in \iso \}$ by Theorem \ref{thm:isolated_orthogonality_implies_Relative_hyperbolicity}. For each $W \sqcup W^c \in \iso$, the product region $P_{W \sqcup W^c}$ is quasi-isometric to the product  $CW \times CW^c$ by the distance formula and the fact that $W$ and $W^c$ are $\nest$--minimal in $\mf{S}$. Since $W$ and $W^c$ are both connected for any $W \sqcup W^c \in \iso$, the curve graphs $CW$ and $CW^c$ will both have infinite diameter.
\end{proof}

\subsection{A conjecture on classifying relative hyperbolicity}\label{subsec:conjecture}

Since the HHS structure in Theorem \ref{thm:multicurve_graphs_are_HHS} has clean containers, the work in Section \ref{sec:clean_hhgs} can be adapted to show that if a hierarchical, twist free graph of multicurves is relatively hyperbolic, then it admits an HHS structure with isolated orthogonality. However, this HHS structure will be markedly different from the one described in Theorem \ref{thm:multicurve_graphs_are_HHS}. Thus, Theorem \ref{thm:isolated_orthogonality_implies_Relative_hyperbolicity} only provides a sufficient condition for relative hyperbolicity of such graphs in terms of the set of witnesses.  This inspires the following conjecture as a companion to Corollary \ref{cor:hyperbolic_graphs_of_multicurves}.

\begin{conj}\label{conj:unique_dijoint_pairs}
A hierarchical, twist free graph of multicurves $\mc{G}$ is relatively hyperbolic if and only if  the HHS structure for $\mc{G}$ from Theorem \ref{thm:multicurve_graphs_are_HHS} has isolated orthogonality.
\end{conj}

We now present some evidence for Conjecture \ref{conj:unique_dijoint_pairs} by noting it holds for the separating curve graph, the pants graph, and the cut system graph. In what follows, we say a metric space is \emph{not relatively hyperbolic} if there doesn't not exist a collection of subsets $\mc{P}$ that satisfies Definition \ref{defn:relatively_hyperbolic_space}. 

That the pants graph satisfies Conjecture \ref{conj:unique_dijoint_pairs} is known in the literature by combining results of Brock, Farb, Masur, Behrstock, Drutu, and Mosher.
\begin{thm}
	Let $S=S_{g,n}$ and $\mf{S}$ be the HHS structure from Theorem \ref{thm:multicurve_graphs_are_HHS} for $\mc{P}(S)$.  \begin{itemize}
		\item If $\xi(S) \leq 2$, then $\mc{P}(S)$ is hyperbolic and $\mf{S}$ contains no disjoint subsurfaces \cite{Brock_Farb_Rank}.
		\item If $\xi(S) =3$, then $\mc{P}(S)$ is relatively hyperbolic \cite{Brock_Masur_WP_Low_complexity} and $\mf{S}$ has isolated orthogonality \emph{[Theorem \ref{thm:rel_hyp_complexes_of_curves}]}.
		\item If $\xi(S) \geq 4$, then $\mc{P}(S)$ is not relatively hyperbolic \cite{behrstock_Drutu_Mosher_Thickness,Brock_Masur_WP_Low_complexity}.
	\end{itemize} 
\end{thm}

For the surfaces where $\Sep$ is connected, Conjecture \ref{conj:unique_dijoint_pairs} holds by combining Theorem \ref{thm:rel_hyp_complexes_of_curves},  Corollary \ref{cor:hyperbolic_graphs_of_multicurves}, and joint work of the author and Vokes \cite[Theorem 1]{russell_vokes}.

	\begin{thm}
		Let $S=S_{g,n}$ and $\mf{S}$ be the HHS structure from Theorem \ref{thm:multicurve_graphs_are_HHS} for $\Sep$. Assume $2g+n \geq 5$ and $S \neq S_{2,1}$. \begin{itemize}
			\item If $n\geq 3$, then $\Sep$ is hyperbolic and $\mf{S}$ contains no disjoint subsurfaces \cite{vokes}.
			\item If $n=0$ or $2$, then $\Sep$ is relatively hyperbolic and $\mf{S}$ has isolated orthogonality \emph{[Theorem \ref{thm:rel_hyp_complexes_of_curves}]}.
			\item If $n=1$, then $\Sep$ is not relatively hyperbolic \cite{russell_vokes}.
		\end{itemize} 
	\end{thm}

The separating curve graph is non-empty, but disconnected for $S_{0,4}$, $S_{1,2}$, $S_{2,0}$, and $S_{2,1}$. In these cases, the edge relation can be modified to achieve a connected, and hence hierarchical, graph by adding edges between separating curves which intersect at most $4$ times .  In these exceptional cases, the separating curve graph (with this alternative edge condition) has no disjoint witnesses and is thus hyperbolic by  Corollary \ref{cor:hyperbolic_graphs_of_multicurves}.

The cut system graph has no disjoint witnesses when $g=1$ and isolated orthogonality when $g=2$. Thus, it suffices to check that the cut system graph is not relatively hyperbolic for $g \geq 3$. This will follow by applying  Proposition 7.2 of \cite{RST_Quasiconvexity}, a general result about hierarchically hyperbolic spaces that can by restated in this context as follows.

\begin{prop}\label{prop:all_qc_subsets_are_hyperbolic}
	Let $\mc{G}$ be a hierarchical and twist-free graph of multicurves that is not hyperbolic. If for all $U,V \in \Wit(\mc{G})$ with $U^c,V^c \in \Wit(\mc{G})$, there exist $U=W_0,W_1,\dots,W_n = V$ so that each $W_i \in \Wit(\mc{G})$ and $W_i$ is disjoint from $W_{i+1}$ for all $0\leq i <n$, then $\mc{G}$ is not relatively hyperbolic.
\end{prop}

\begin{proof}
	Let $\mf{S}$ be the HHS structure for $\mc{G}$ from Theorem \ref{thm:multicurve_graphs_are_HHS}.	
	The two requirements of Proposition 7.2 of \cite{RST_Quasiconvexity} are satisfied for $\mf{S}$: the first is satisfied since the mapping class group of $S$ acts coboundedly on $CS$ and the second is satisfied by the  hypotheses on $\Wit(\mc{G})$.  Proposition 7.2 of \cite{RST_Quasiconvexity} therefore says all strongly quasiconvex subsets of $\mc{G}$ are hyperbolic. Hence, $\mc{G}$  cannot be relatively hyperbolic, as the peripheral subsets of a relatively hyperbolic space are strongly quasiconvex and a space that is hyperbolic relative to hyperbolic peripherals is itself hyperbolic.
\end{proof}

\begin{thm}
	Let $S=S_{g,0}$ and $\mf{S}$ be the HHS structure from Theorem \ref{thm:multicurve_graphs_are_HHS} for $\cut(S)$. 
	\begin{itemize}
		\item If $g=1$, then $\cut(S)$ is hyperbolic and $\mf{S}$ contains no disjoint subsurfaces.
		\item If $g = 2$, then $\cut(S)$ is relatively hyperbolic and $\mf{S}$ has isolated orthogonality.
		\item If $g\geq 3$, then $\cut(S)$ is not relatively hyperbolic.
	\end{itemize} 
\end{thm}

\begin{proof}
	As noted before Proposition \ref{prop:all_qc_subsets_are_hyperbolic}, it suffices to show that $\cut(S)$ is not relatively hyperbolic if $g\geq 3$. We shall establish this by showing that $\Wit(\cut(S))$ satisfies the conditions of Proposition \ref{prop:all_qc_subsets_are_hyperbolic}. Assume $g \geq 3$ and let $U,V \in \Wit(\cut(S))$ with $U^c,V^c \in \Wit(\cut(S))$. Recall, $\Wit(\cut(S))$ is the collection of all connected subsurfaces containing genus.
	Thus, there exist separating curves $\beta_U \subseteq U$ and $\beta_V \subseteq V$. Since $g\geq 3$, there exists a sequence of sequentially disjoint separating curves $\beta_U=\alpha_1,\dots,\alpha_m = \beta_V$ in $\Sep$. The following claim completes the proof by showing $\alpha_1,\dots,\alpha_m$ can be promoted to a sequence $W = W_0, \dots, W_n=Z$ of sequentially disjoint witnesses of $\cut(S)$.
	\end{proof}
	
	\textbf{Claim.}
	If $\alpha$ and $\beta$ are disjoint separating curves on $S$ and $W,Z\in \Wit(\cut(S))$ with $ \alpha \in \partial W$ and $\beta \in \partial Z$, then there exists a sequence of sequentially disjoint subsurfaces $W =  W_1,\dots, W_k = Z$ with $k \leq 5$ and each $W_i \in \Wit(\cut(S))$.
	
	\begin{proof}
	    The claim is automatically satisfied if $W$ and $Z$ are disjoint, so we shall assume this is not the case. Since $\alpha$ and $\beta$ are disjoint, we must have either $\beta \cap W = \emptyset$ or $\beta \subseteq W$. If $\beta \cap W = \emptyset$, then there exists a component $Y$ of  $S-\beta$ with $Y \subseteq W^c$ and $W$, $Y$, $Z$ is the desired sequence. If $\beta \subseteq W$, then there exists a component $Y_1$ of $S-\alpha$ that is contained in $W^c$ and disjoint from $\beta$. Let $Y_2$ be the component of $S-\beta$ that is disjoint from $Y_1$. Since $\beta$ is separating and contained in $\partial Z$, either $Y_2$ is disjoint from $Z$ and the desired sequence is $W$, $Y_1$, $Y_2$, $Z$, or $Y_2^c$ is disjoint from $Z$ and the desired sequence is $W, Y_1, Y_2, Y_2^c, Z$.  As $\Wit(\cut(S))$ is the collection of subsurfaces containing genus and $S$ is closed, each of the above sequences contains only elements of $\Wit(\cut(S))$.
	\end{proof}

\bibliography{Russell}{}
\bibliographystyle{alpha}

\end{document}